\documentclass[11pt,letterpaper]{amsart}
\usepackage[foot]{amsaddr}

\usepackage{mathtools}
\usepackage{amsmath}
\usepackage[T1]{fontenc}
\usepackage[latin9]{inputenc}
\usepackage{geometry}
\usepackage{wrapfig}
\usepackage{multicol}
\usepackage{graphicx}
\usepackage{soul}
\usepackage{xcolor}
\usepackage{amssymb}
\usepackage{placeins}
\usepackage{bbm}
\usepackage{cite}
\usepackage{caption}
\usepackage{enumerate}
\usepackage{afterpage}
\usepackage{enumitem}
\usepackage{bmpsize}
\usepackage{hyperref}
\usepackage{tabu}
\usepackage{enumitem}
\numberwithin{equation}{section}
\usepackage{stmaryrd}
\usepackage{tikz}
\usetikzlibrary{matrix,graphs,arrows,positioning,calc,decorations.markings,shapes.symbols,patterns, decorations.pathreplacing}

\usepackage{ifthen}
\usetikzlibrary{math}
\usetikzlibrary{matrix,graphs,arrows,positioning,calc,decorations.markings,shapes.symbols}

\makeatletter
\renewcommand{\email}[2][]{%
  \ifx\emails\@empty\relax\else{\g@addto@macro\emails{,\space}}\fi%
  \@ifnotempty{#1}{\g@addto@macro\emails{\textrm{(#1)}\space}}%
  \g@addto@macro\emails{#2}%
}
\makeatother

\newtheorem{theorem}{Theorem}[section]
\newtheorem{lemma}[theorem]{Lemma}
\newtheorem{proposition}[theorem]{Proposition}

\newtheorem{corollary}[theorem]{Corollary}
{ \theoremstyle{definition}
}
{ \theoremstyle{remark}
\newtheorem{remark}[theorem]{Remark}}

\newcommand{\carry}{CARRY}
\newcommand{\rsk}{\mathsf{RSK}}
\newcommand{\rskhs}{\mathsf{RSK}^{\mathsf{sym}}}

\usepackage{ifthen}

\title{A note on last passage percolation and Schur processes}
\date{\today}
\author{Evgeni Dimitrov} 
\author{Zongrui Yang} 
\begin{document}

\begin{abstract}
In this note we provide a short proof of the distributional equality between last passage percolation with geometric weights along a general down-right path and Schur processes. We do this in both the full-space and half-space settings, and for general parameters. The main inputs for our arguments are generalizations of the Robinson-Schensted-Knuth correspondence and Greene's theorem due to Krattenthaler, which are based on Fomin's growth diagrams. 
\end{abstract}

\maketitle
\tableofcontents

%
%
\section{Introduction}\label{Section1} 
{\em Last passage percolation} (LPP) and the related models known as the {\em polynuclear growth model} (PNG), and the {\em totally asymmetric simple exclusion process} (TASEP), are some of the most well-studied stochastic systems in the {\em Kardar-Parisi-Zhang} (KPZ) universality class. Their analysis has revealed many remarkable connections between symmetric functions, increasing subsequences of random permutations, the combinatorics of Young tableaux, and related topics; see, e.g. \cite[Chapters 4--5]{Romik15}. 

There are different versions of LPP depending on the nature of the background noise $(w_{ij}:i,j \geq 1)$ that defines the model, which could be independent Bernoulli, geometric or exponential random variables, and in this paper we focus on the geometric case. In the seminal paper \cite{J00} Johansson proved the convergence of geometric LPP to the Tracy-Widom distribution, and hence demonstrated that it belongs to the KPZ universality class. In \cite{BR01c}, building up on the earlier works \cite{BR01a, BR01b}, Baik and Rains considered a symmetrized version of geometric LPP, corresponding to conditioning on $w_{ij} = w_{ji}$ for $i \neq j$, and also established various convergence results to analogues of the Tracy-Widom distribution. Following \cite{BBCS}, we will refer to the LPP with independent weights as {\em full-space}, and the one with symmetrized weights as {\em half-space}.

An important reason behind the success in the asymptotic analysis of LPP is that it has the structure of a {\em determinantal point process} in the full-space and a {\em Pfaffian point process} in the half-space setting. These structural properties come as consequences from the statements that there is a distributional equality between the full-space LPP and the {\em Schur processes} from \cite{OR03} and between the half-space LPP and the {\em Pfaffian Schur processes} from \cite{BR05}. The connections between LPP and Schur processes have been well-known to experts, and come from two deep results in combinatorics: the {\em Robinson-Schensted-Knuth} (RSK) {\em correspondence}, see \cite{K70} for the original result and \cite[Section 7.11]{Stanley99} for a textbook treatment, and {\em Greene's theorem} \cite{Greene74}. Despite being recognized over two decades ago, the connections between RSK and LPP, and their various generalizations, are still actively being studied and used to uncover many remarkable structural properties for a large family of stochastic systems, see \cite{DNV22} and the references therein. 

The purpose of this note is to provide a short derivation of the distributional equality between LPP and Schur processes. For the sake of completeness we do this in both the full-space and half-space settings (see Theorems \ref{T1} and \ref{T2}). These results are likely known to experts, but to our knowledge they have not appeared previously in this form. Existing references usually state related results in a way that is hard to parse and short on details, which was a key motivation for writing this note. We have aimed to state the results of the paper under the most general conditions we are aware of, allowing for arbitrary down-right paths and parameters. A central goal has been to make the arguments accessible even to non-experts and essentially self-contained --- the only prerequisites are the RSK algorithm and Greene's theorem from \cite[Section 4.1]{K06}.

%
%
\section{Definitions and main results}\label{Section2}

%
%
\subsection{Full-space geometric LPP}\label{Section2.1} The full-space model depends on two sequences of real parameters $\{x_i\}_{i \geq 1}$, $\{y_i\}_{i \geq 1}$, such that $x_i, y_j \geq 0$ and $x_iy_j \in [0,1)$ for all $i,j \geq 1$. The background noise is given by an array $W = (w_{i,j}:i,j \geq 1)$ of independent geometric variables $w_{i,j} \sim \mathrm{Geom}(x_iy_j)$, i.e.
\begin{equation}\label{Eq.Geom}
\mathbb{P}(w_{i,j} = k) = (x_iy_j)^k \cdot (1 - x_i y_j) \mbox{ for } k \in \mathbb{Z}_{\geq 0}.
\end{equation}
We visualize the weight $w_{i,j}$ as being associated with the vertex $(i,j)$ on the lattice $\mathbb{Z}^2$, see the left side of Figure \ref{Fig.Grid}. An {\em up-right path} $\pi$ in $\mathbb{Z}^2$ is a (possibly empty) sequence of vertices $\pi = (v_1, \dots, v_r)$ with $v_i \in \mathbb{Z}^2$ and $v_i - v_{i-1} \in \{(0,1), (1,0)\}$. For an up-right path $\pi$ in $\mathbb{Z}_{\geq 1}^2$, we define its {\em weight} by
\begin{equation}\label{EPathWeight}
W(\pi) = \sum_{v\in\pi} w_{v},
\end{equation}
and for any $(m,n) \in \mathbb{Z}_{\geq 1}^2$ we define the {\em last passage time} $G_1(m,n)$ by
\begin{equation}\label{Eq.LPT}
G_1(m,n) = \max_{\pi} W(\pi),
\end{equation}
where the maximum is over all up-right paths that connect $(1,1)$ with $(m,n)$. It turns out that $G_1(m,n)$ can naturally be embedded into a sequence $G(m,n) = (G_k(m,n): k \geq 1)$ of higher-rank last passage times. There are two equivalent ways of introducing these, which we describe next.

Suppose that $A = (a_{i,j}: i =1, \dots, m, j = 1, \dots, n)$ is an $n\times m$ matrix with {\em non-negative} real entries. To make our notation consistent with the one for the array $W$ above, we number the rows of $A$ by $j = 1,\dots, n$ from bottom to top, and the columns of $A$ by $i = 1, \dots, m$ from left to right. For $k =1, \dots, \min (m,n)$, we define 
\begin{equation}\label{Eq.GDef}
g_k(A) = \max_{\pi_1, \dots, \pi_k} [A(\pi_1) + \cdots + A(\pi_k)], \mbox{ where } A(\pi) = \sum_{v \in \pi} a_{v},
\end{equation}
and the maximum is over $k$-tuples of pairwise disjoint up-right paths $(\pi_1, \dots, \pi_k)$, with $\pi_i$ connecting the vertex $(1,i)$ to $(m, n-k+i)$. Note that $g_{\min(m,n)}(A) = \sum_{i = 1}^m \sum_{j = 1}^n a_{i,j}$, as one can find $\min(m,n)$ pairwise disjoint paths as above that contain all the vertices, see the right side of Figure \ref{Fig.Grid}. When $k \geq \min(m,n) +1$ we cannot find $k$ pairwise disjoint paths as above, and the convention is to set
\begin{equation}\label{Eq.GDef2}
g_k(A) = \sum_{i = 1}^m \sum_{j = 1}^n a_{i,j} \mbox{ for } k \geq \min(m,n) + 1.
\end{equation}

\begin{figure}[ht]
\centering
\begin{tikzpicture}[scale=0.9]

\def\m{7} 
\def\n{5} 
\def\k{3} 

\definecolor{Bg}{gray}{1.0}        
\definecolor{C1}{gray}{0.8} 
\definecolor{C2}{gray}{0.5} 
\definecolor{C3}{gray}{0.2} 

\begin{scope}[shift={(0,0)}]


  \foreach \j in {1,...,\m}{
      \node at (\j+0.5, -0.5) {\(\j\)};
  }
  \foreach \j in {1,...,\n}{
      \node at (0.35,-0.5 + \j) {\(\j\)};
  }

  \foreach \x/\y in {1/2, 1/3, 2/3, 3/3, 3/4, 3/5, 4/5} {
  \fill[C1] (\x,\y) rectangle ++(1,-1);
  
}

  \foreach \x/\y in {2/1, 3/1, 4/3, 6/3, 6/4, 7/5} {
  \fill[C2] (\x,\y) rectangle ++(1,-1);
  
}
\foreach \i in {1,...,\n}{
  \foreach \j in {1,...,\m}{
    \draw[black] (\j,\i) rectangle ++(1,-1);
    \node at (\j+0.5, \i-0.5) {$w_{\j, \i}$};
  }
}

\node at (\m+1.5,-0.5) {$i$};
\node at (0.35,-0.5 + \n + 1) {$j$};

\begin{scope}[shift={(1,-1)}]
  \draw[fill=C1] (3,6.75) rectangle ++(0.5,-0.5);
  \node[anchor=west] at (3.5,6.5) {$\pi$};
  \draw[fill=C2] (4.5,6.75) rectangle ++(0.5,-0.5);
  \node[anchor=west] at (5,6.5) {$\chi$};
\end{scope}
\end{scope}

\begin{scope}[shift={(9,0)}]

  \foreach \x/\y in {1/1, 2/1, 3/1, 4/1, 5/1, 6/1, 7/1} {
  \fill[C1] (\x/2,\y/2) rectangle ++(1/2,-1/2);
  
}

  \foreach \x/\y in {1/2, 2/2, 3/2, 3/2, 4/2, 5/2, 6/2, 7/2} {
  \fill[C2] (\x/2,\y/2) rectangle ++(1/2,-1/2);
  
}

  \foreach \x/\y in {1/3, 2/3, 3/3, 3/3, 4/3, 5/3, 6/3, 7/3} {
  \fill[C3] (\x/2,\y/2) rectangle ++(1/2,-1/2);
  
}

\foreach \x/\y in {1/1, 2/1, 3/1, 3/2, 3/3, 3/4, 3/5} {
  \fill[C1] (5+ \x/2,\y/2) rectangle ++(1/2,-1/2);
  
}

  \foreach \x/\y in {1/2, 2/2, 2/3, 2/4, 2/5, 2/6, 3/6} {
  \fill[C2] (5+ \x/2,\y/2) rectangle ++(1/2,-1/2);
  
}

  \foreach \x/\y in {1/3, 1/4, 1/5, 1/6, 1/7, 2/7, 3/7} {
  \fill[C3] (5+ \x/2,\y/2) rectangle ++(1/2,-1/2);
  
}

\foreach \i in {1,...,3}{
  \foreach \j in {1,...,7}{
    \draw[black] (\j/2,\i/2) rectangle ++(1/2,-1/2);
    \draw[black] (\i/2 +5, \j/2) rectangle ++(1/2,-1/2);
  }
}

\begin{scope}[shift={(1,-1)}]
  \draw[fill=C1] (1,5.5) rectangle ++(0.5,-0.5);
  \node[anchor=west] at (1.5,5.25) {$\pi_1$};
  \draw[fill=C2] (2.5,5.5) rectangle ++(0.5,-0.5);
  \node[anchor=west] at (3,5.25) {$\pi_2$};
  \draw[fill=C3] (4,5.5) rectangle ++(0.5,-0.5);
  \node[anchor=west] at (4.5, 5.25) {$\pi_3$};
\end{scope}
\end{scope}

\end{tikzpicture}
\caption{The left side depicts the array $W = (w_{i,j}: i,j \geq 1)$, an up-right path $\pi$ and a NE-chain $\chi$. The vertices in $\mathbb{Z}^2$ correspond to the midpoints of the squares and are not drawn. The right side depicts $\min(m,n)$ pairwise disjoint up-right paths, with $\pi_i$ connecting $(1,i)$ to $(m, n-k+i)$ that cover the whole $n \times m$ rectangle, for $(m,n) = (7,3)$ and $(m,n) = (3,7)$.} \label{Fig.Grid}
\end{figure}

A {\em north-east (NE)-chain} $\chi = ((a_1,b_1), \dots, (a_r, b_r))$ is a (possibly empty) sequence with $a_i,b_i \in \mathbb{Z}$, and $a_{i-1} \leq a_i$, $b_{i-1} \leq b_i$, $(a_{i-1}, b_{i-1}) \neq (a_i, b_i)$. Analogously to (\ref{Eq.GDef}), we define for $k \geq 1$ 
\begin{equation}\label{Eq.HDef}
h_k(A) = \max_{\chi_1, \dots, \chi_k} [A(\chi_1) + \cdots + A(\chi_k)], \mbox{ where } A(\chi) = \sum_{v \in \chi} a_{v},
\end{equation}
and the maximum is over $k$-tuples of pairwise disjoint NE-chains. As up-right paths are in particular NE-chains, we have the following trivial inequalities for any real matrix $A$ and $k \geq 1$
\begin{equation}\label{Eq.GHIneq}
g_k(A) \leq h_k(A).
\end{equation}
When the entries of $A$ are non-negative we have the following improvement.
\begin{lemma}\label{L.GequalsH} Suppose that $A$ is a real $n\times m$ matrix with non-negative entries. If $g_k(A)$ are as in (\ref{Eq.GDef}) and (\ref{Eq.GDef2}), and $h_k(A)$ are as in (\ref{Eq.HDef}), then  
\begin{equation}\label{Eq.GHEqual}
g_k(A) = h_k(A) \mbox{ for all } k \geq 1.
\end{equation}
\end{lemma}
\begin{remark}\label{R.S21R1} Lemma \ref{L.GequalsH} appears to be known to experts; however, as we could not find its proof in the literature, we provide it in Appendix \ref{SectionA}.
\end{remark}
With the above notation, we can now define the higher-rank last passage times by
\begin{equation}\label{Eq.LPTk}
G_k(m,n) := g_k(W[m,n]) = h_k(W[m,n]),
\end{equation}
where $k \geq 1$, $W[m,n] = (w_{i,j}: i = 1, \dots, m, j = 1, \dots, n)$, and the last equality follows by Lemma \ref{L.GequalsH}.

%
%
\subsection{Half-space geometric LPP}\label{Section2.2} The half-space model depends on a sequence of real parameters $\{x_i\}_{i \geq 1}$ and a parameter $c$ such that $x_i \geq 0$, $c \geq 0$, $x_ix_j \in [0,1)$, and $cx_i \in [0,1)$. The background noise is given by an array $W = (w_{i,j}: i,j \geq 1)$, where $(w_{i,j}: 1\leq i \leq j)$ are independent geometric variables with $w_{i,j} \sim \mathrm{Geom}(x_ix_j)$ when $i \neq j$ and $w_{i,i} \sim \mathrm{Geom}(cx_i)$, and $w_{i,j} = w_{j,i}$ for all $i,j \geq 1$.

In this setup we introduce the same higher-rank last passage times as in (\ref{Eq.LPTk}). One thing to notice here is that for the half-space model
\begin{equation}\label{Eq.SymmG}
G_k(m,n) = G_k(n,m) \mbox{ for all } k,m,n \geq 1.
\end{equation}
The latter follows from the statement $h_k(A^T) = h_k(A)$ for any real matrix $A$ (here $A^T$ denotes the transpose of $A$), which is immediate from the definition of NE-chains.

%
%
\subsection{Main results}\label{Section2.3} From the definition of NE-chains and (\ref{Eq.HDef}) one readily observes that if we set
\begin{equation}\label{Eq.DefLambda}
\lambda_{1}(m,n) = G_1(m,n) \mbox{ and } \lambda_k(m,n) = G_k(m,n) - G_{k-1}(m,n) \mbox{ for } k \geq 2,
\end{equation}
then $\lambda_k(m,n)$ are non-negative integers, $\lambda_k(m,n) = 0$ for $k > \min(m,n)$ and $\sum_{i = 1}^{\infty} \lambda_i(m,n) = \sum_{i = 1}^m \sum_{j = 1}^n w_{i,j}$. Additionally, one has that 
\begin{equation}\label{Eq.LambdaDec}
\lambda_{k}(m,n) \geq \lambda_{k+1}(m,n) \mbox{ for all } k \geq 1,
\end{equation}
which means that $\lambda(m,n):=(\lambda_{k}(m,n): k \geq 1)$ is a {\em partition} (a decreasing sequence of non-negative integers that is eventually zero). The latter follows from the following more general statement, which will be established in Section \ref{Section3.1} as an immediate consequence of the RSK correspondence and Greene's theorem in that section.
\begin{lemma}\label{L.HRSK} Suppose that $A$ is a real $n\times m$ matrix with non-negative entries. If $h_k(A)$ are as in (\ref{Eq.HDef}) and $h_0(A) = 0$, then  
\begin{equation}\label{Eq.HDec}
h_k(A) - h_{k-1}(A) \geq h_{k+1}(A) - h_k(A) \mbox{ for all } k \geq 1.
\end{equation}
\end{lemma}

The results we present describe the joint distribution of $(\lambda(v_1), \dots, \lambda(v_r))$ for several $v_1, \dots, v_r \in \mathbb{Z}_{\geq 1}^2$, and in order to state them we require more notation. 

A {\em down-right path} $\gamma$ is a sequence $\gamma = (v_0, \dots, v_{r})$ where the vertices satisfy $v_i - v_{i-1} \in \{(0,-1), (1,0)\}$ for $i = 1,\dots, r$. Note that any down-right path $\gamma$ can be identified with a word of length $r$, where the $i$-th letter is $D$ if $v_{i}-v_{i-1} = (0,-1)$ and is $R$ if $v_{i} - v_{i-1} = (1,0)$. If $\gamma = (v_0, \dots, v_{m+n})$ is a down-right path, with $v_0 = (0,n)$ and $v_{m+n} = (m,0)$, we define the set
\begin{equation}\label{Eq.FerresShapeDef}
Y(\gamma) = \left\{(i,j) \in \mathbb{Z}_{\geq 1}^2: (i+d, j+d) \in \gamma \mbox{ for some } d \geq 0 \right\},
\end{equation}
which can be visualized as the Ferrers shape, or equivalently Young diagram, in the first quadrant enclosed by the path $\gamma$, see the left side of Figure \ref{Fig.DownRightPath}.

\begin{figure}[tb]
\centering
\begin{tikzpicture}[scale=0.65]

\def\m{7} 
\def\mm{6} 
\def\n{5} 
\def\nn{4} 
\def\k{3} 

\definecolor{Bg}{gray}{1.0}        
\definecolor{C1}{gray}{0.8} 
\definecolor{C2}{gray}{0.5} 
\definecolor{C3}{gray}{0.2} 

\begin{scope}[shift={(0,0)}]


  \foreach \j in {0,...,\mm}{
      \node at (\j+1.5, -0.5) {\(\j\)};
  }
  \foreach \j in {0,...,\nn}{
      \node at (0.35,-0.5 + \j+1) {\(\j\)};
  }

  \foreach \x/\y in {1/5, 2/5, 3/5, 3/4, 4/4, 4/3, 4/2, 5/2, 6/2, 7/2, 7/1} {
  \fill[C2] (\x,\y) rectangle ++(1,-1);
  
}

\foreach \i in {1,...,\n}{
  \foreach \j in {1,...,\m}{
    \draw[black] (\j,\i) rectangle ++(1,-1);
  }
}

  \draw[ line width=3pt]
    (2,1) -- (8, 1) -- (8,2) -- (5,2) -- (5,4) -- (4,4) -- (4,5) -- (2, 5) -- (2,1) ;

\end{scope}

\begin{scope}[shift={(10,0)}]


  \foreach \j in {0,...,\mm}{
      \node at (\j+1.5, -0.5) {\(\j\)};
  }
  \foreach \j in {0,...,\nn}{
      \node at (0.35,-0.5 + \j+1) {\(\j\)};
  }

  \foreach \x/\y in {1/5, 2/5, 3/5, 3/4, 3/3, 4/3, 4/2, 5/2, 6/2, 7/2, 7/1} {
  \fill[C2] (\x,\y) rectangle ++(1,-1);
  
}

\node at (4.5,3.5) {$*$};

\foreach \i in {1,...,\n}{
  \foreach \j in {1,...,\m}{
    \draw[black] (\j,\i) rectangle ++(1,-1);
  }
}

  \draw[ line width=3pt]
    (2,1) -- (8, 1) -- (8,2) -- (5,2) -- (5,3) -- (4,3) -- (4,5) -- (2, 5) -- (2,1) ;

\end{scope}

\end{tikzpicture}
\caption{The left side depicts a down-right path $\gamma$ that connects $(0,n)$ to $(m,0)$ for $n = 4$, $m = 6$. The set $Y(\gamma)$ in (\ref{Eq.FerresShapeDef}) is enclosed by bold black lines. The word corresponding to $\gamma$ is $RRDRDDRRRD$. The right side depicts the down-right path $\gamma'$ obtained from $\gamma$ by replacing the second $RD$ in the word corresponding to $\gamma$ with $DR$. The star indicates the unique vertex in $Y(\gamma) \setminus Y(\gamma')$.} \label{Fig.DownRightPath}
\end{figure}

Given two partitions $\lambda, \mu$, we say that they {\em interlace} (denoted $\lambda \succeq \mu$ or $\mu \preceq \lambda$) if $\lambda_1 \geq \mu_1 \geq \lambda_2 \geq \mu_2 \geq  \cdots$. For two partitions $\lambda, \mu$ the {\em skew Schur polynomial} in a single variable $x$ is given by
\begin{equation}\label{Eq.DefSchur}
s_{\lambda/\mu}(x) = {\bf 1}\{\lambda \succeq \mu\} \cdot x^{\sum_{i \geq 1} (\lambda_i - \mu_i)}. 
\end{equation}
We denote the {\em zero} (or {\em empty}) partition by $\emptyset$.

With the above notation in place we can state the main result about the full-space model, see also Figure \ref{Fig.T1}.
\begin{theorem}\label{T1} Suppose that $W$ is as in Section \ref{Section2.1}, and $(\lambda(m,n): m,n \geq 1)$ are as in (\ref{Eq.DefLambda}) for $G_k(m,n)$ as in (\ref{Eq.LPTk}). In addition, set $\lambda(m,n) = \emptyset$ when $m = 0$ or $n = 0$. Finally, fix any $M, N \geq 0$ and a down-right path $\gamma = (v_0, \dots, v_{M+N})$ from $(0,N)$ to $(M,0)$, whose corresponding word is $s_1s_2\cdots s_{M+N}$. Then, for any partitions $\lambda^0, \dots, \lambda^{M+N}$ we have that 
\begin{equation}\label{Eq.FSDistr}
\begin{split}
&\mathbb{P}(\lambda(v_i) = \lambda^i \mbox{ for } i = 0, \dots, M+N) \\
&= Z \cdot {\bf 1}\left\{\lambda^0 = \lambda^{M+N} = \emptyset \right\} \cdot \prod_{i = 1}^{M+N} Q(\lambda^{i-1}, \lambda^{i}),
\end{split}
\end{equation}
where $Z = \prod_{(a,b) \in Y(\gamma)} (1-x_ay_b)$, and
\begin{equation}\label{Eq.DefQ}
Q(\lambda^{i-1}, \lambda^{i})=
\begin{cases}
s_{\lambda^{i-1}/\lambda^{i}}(y_t),
& \text{if } v_{i-1}\in \mathbb Z\times\{t\}
\text{ and } s_i=D,
\\
s_{\lambda^{i}/\lambda^{i-1}}(x_t),
& \text{if } v_{i-1}\in \{t-1\}\times\mathbb Z
\text{ and } s_i=R.
\end{cases}
\end{equation}

\end{theorem}

\begin{figure}[h]
\centering
\begin{tikzpicture}[scale=0.83]
\begin{scope}[shift={(0,3)}, scale = 0.8]
\def\m{6} 
\def\mm{5} 
\def\n{4} 
\def\nn{3} 
\def\k{3} 
\definecolor{Bg}{gray}{1.0}   
\definecolor{C1}{gray}{0.8} 
\definecolor{C2}{gray}{0.5} 
\definecolor{C3}{gray}{0.2} 
\foreach \x/\y in {1/4, 2/4, 3/4, 4/4, 4/3, 5/3, 5/2, 6/2, 6/1} {
\fill[C2] (\x,\y) rectangle ++(1,-1);}
\foreach \i in {1,...,\n}{
  \foreach \j in {1,...,\m}{
    \draw[black] (\j,\i) rectangle ++(1,-1);
  }
} 
\node at (2.5,1.5) {\(1\)};
\node at (3.5,1.5) {\(0\)};
\node at (4.5,1.5) {\(1\)};
\node at (5.5,1.5) {\(2\)};
\node at (6.5,1.5) {\(1\)};
\node at (2.5,2.5) {\(0\)};
\node at (3.5,2.5) {\(2\)};
\node at (4.5,2.5) {\(3\)};
\node at (5.5,2.5) {\(1\)};
\node at (6.5,2.5) {\(2\)};
\node at (2.5,3.5) {\(3\)};
\node at (3.5,3.5) {\(1\)};
\node at (4.5,3.5) {\(0\)};
\node at (5.5,3.5) {\(2\)};
\node at (6.5,3.5) {\(4\)};

\node at (4,-0.5) {\(W\)};
\end{scope}

\begin{scope}[shift={(-0.5,-1.7)}, scale = 0.8]
\def\m{6} 
\def\mm{5} 
\def\n{4} 
\def\nn{3} 
\def\k{3} 

\definecolor{Bg}{gray}{1.0}   
\definecolor{C1}{gray}{0.8} 
\definecolor{C2}{gray}{0.5} 
\definecolor{C3}{gray}{0.2} 
\foreach \i in {2,...,\n}{
  \foreach \j in {2,...,\m}{
    \draw[black] (\j,\i) rectangle ++(1,-1);
  }
} 
\node at (2.5,1.5) {\(1\)};
\node at (3.5,1.5) {\(1\)};
\node at (4.5,1.5) {\(2\)};
\node at (5.5,1.5) {\(4\)};
\node at (6.5,1.5) {\(5\)};
\node at (2.5,2.5) {\(1\)};
\node at (3.5,2.5) {\(3\)};
\node at (4.5,2.5) {\(6\)};
\node at (5.5,2.5) {\(7\)};
\node at (6.5,2.5) {\(9\)};
\node at (2.5,3.5) {\(4\)};
\node at (3.5,3.5) {\(5\)};
\node at (4.5,3.5) {\(6\)};
\node at (5.5,3.5) {\(9\)};
\node at (6.5,3.5) {\(13\)};

\node at (1.5,2.5) {\(G_1\)};
\end{scope}

\begin{scope}[shift={(9.5,-1.7)}, scale = 0.8]
\def\m{6} 
\def\mm{5} 
\def\n{4} 
\def\nn{3} 
\def\k{3} 

\definecolor{Bg}{gray}{1.0}   
\definecolor{C1}{gray}{0.8} 
\definecolor{C2}{gray}{0.5} 
\definecolor{C3}{gray}{0.2} 
\foreach \i in {2,...,\n}{
  \foreach \j in {2,...,\m}{
    \draw[black] (\j,\i) rectangle ++(1,-1);
  }
} 
\node at (2.5,1.5) {\(1\)};
\node at (3.5,1.5) {\(1\)};
\node at (4.5,1.5) {\(2\)};
\node at (5.5,1.5) {\(4\)};
\node at (6.5,1.5) {\(5\)};
\node at (2.5,2.5) {\(1\)};
\node at (3.5,2.5) {\(3\)};
\node at (4.5,2.5) {\(7\)};
\node at (5.5,2.5) {\(10\)};
\node at (6.5,2.5) {\(13\)};
\node at (2.5,3.5) {\(4\)};
\node at (3.5,3.5) {\(7\)};
\node at (4.5,3.5) {\(11\)};
\node at (5.5,3.5) {\(16\)};
\node at (6.5,3.5) {\(23\)};

\node at (1.5,2.5) {\(G_3\)};
\end{scope}

\begin{scope}[shift={(4.5,-1.7)}, scale = 0.8]
\def\m{6} 
\def\mm{5} 
\def\n{4} 
\def\nn{3} 
\def\k{3} 

\definecolor{Bg}{gray}{1.0}   
\definecolor{C1}{gray}{0.8} 
\definecolor{C2}{gray}{0.5} 
\definecolor{C3}{gray}{0.2} 
\foreach \i in {2,...,\n}{
  \foreach \j in {2,...,\m}{
    \draw[black] (\j,\i) rectangle ++(1,-1);
  }
} 
\node at (2.5,1.5) {\(1\)};
\node at (3.5,1.5) {\(1\)};
\node at (4.5,1.5) {\(2\)};
\node at (5.5,1.5) {\(4\)};
\node at (6.5,1.5) {\(5\)};
\node at (2.5,2.5) {\(1\)};
\node at (3.5,2.5) {\(3\)};
\node at (4.5,2.5) {\(7\)};
\node at (5.5,2.5) {\(10\)};
\node at (6.5,2.5) {\(13\)};
\node at (2.5,3.5) {\(4\)};
\node at (3.5,3.5) {\(7\)};
\node at (4.5,3.5) {\(10\)};
\node at (5.5,3.5) {\(13\)};
\node at (6.5,3.5) {\(19\)};

\node at (1.5,2.5) {\(G_2\)};
\end{scope}

\begin{scope}[shift={(6.5,0)},scale = 2]
\def\m{6} 
\def\mm{5} 
\def\n{4} 
\def\nn{3} 
\def\k{3} 

\draw[line width=1pt]
    (0,3) -- (1, 3) -- (2,3) -- (3,3) -- (3,2) -- (4,2) -- (4,1) -- (5,1) -- (5, 0);
 
\def\r{0.05}
\draw[fill = black] (0,3) circle [radius=\r];
\draw[fill = black] (1,3) circle [radius=\r];
\draw[fill = black] (2,3) circle [radius=\r];
\draw[fill = black] (3,3) circle [radius=\r];
\draw[fill = black] (3,2) circle [radius=\r];
\draw[fill = black] (4,2) circle [radius=\r];
\draw[fill = black] (4,1) circle [radius=\r];
\draw[fill = black] (5,1) circle [radius=\r];
\draw[fill = black] (5,0) circle [radius=\r];
\node at (0.5,3.18) {\(\preceq\)};
\node at (1.5,3.18) {\(\preceq\)};
\node at (2.5,3.18) {\(\preceq\)};
\node at (3.15,2.5) {\rotatebox{-90}{$\succeq$}};
\node at (3.6,2.18) {\(\preceq\)};
\node at (4.15,1.5) {\rotatebox{-90}{$\succeq$}};
\node at (4.5,1.18) {\(\preceq\)};
\node at (5.15,0.5) {\rotatebox{-90}{$\succeq$}};

\node at (0.5,2.82) {\(R\)};
\node at (1.5,2.82) {\(R\)};
\node at (2.5,2.82) {\(R\)};
\node at (2.85,2.5) {$D$};
\node at (3.6,1.82) {\(R\)};
\node at (3.85,1.5) {$D$};
\node at (4.5,0.82) {\(R\)};
\node at (4.85,0.5) {$D$};

\node at (0,3.18) {\(\emptyset\)};
\node at (1,3.18) {\((4)\)};
\node at (2,3.18) {\((5,2)\)};
\node at (3,3.18) {\((6,4,1)\)};
\node at (3.25,2.18) {\((6,1)\)};
\node at (4,2.18) {\((7,3)\)};
\node at (4.15,1.18) {\((4)\)};
\node at (5,1.18) {\((5)\)};
\node at (5.15,0.18) {\(\emptyset\)};

\end{scope}

\end{tikzpicture}  
\caption{The figure depicts a sample random environment $W$, a down-right path $\gamma$ with corresponding word $RRRDRDRD$, and the higher-rank last passage times $G_k(m,n)$ from (\ref{Eq.LPTk}). It also shows the partitions $\{\lambda(v)\}_{v \in \gamma}$ from Theorem \ref{T1}, as well as the interlacing relations for consecutive partitions along the path $\gamma$ implied by (\ref{Eq.DefSchur}), (\ref{Eq.FSDistr}), and (\ref{Eq.DefQ}).}
\label{Fig.T1}
\end{figure}

Theorem \ref{T1} has appeared previously under various different assumptions on $\gamma$ and the parameters. The closest statement we could find is \cite[Theorem 5.2]{johansson2006random}, which with substantial effort can be seen to be equivalent to Theorem \ref{T1} although it is written with a different notation and proved by different means. A related result connecting RSK and Schur processes can be found in \cite[Theorem 3.2]{betea2018perfect}, although that paper does not discuss LPP.\\

We next turn to the half-space setting, where in view of (\ref{Eq.SymmG}), we have that $\lambda(m,n) = \lambda(n,m)$ for all $m, n \geq 1$. Consequently, it suffices to describe the joint distribution of $\lambda(m,n)$ along a down-right path contained in the region  $m \hspace{-1mm} \geq \hspace{-1mm} n$. The exact statement is as follows, see also Figure \ref{Fig.T2}.

\begin{theorem}\label{T2} Suppose that $W$ is as in Section \ref{Section2.2}, and $(\lambda(m,n): m,n \geq 1)$ are as in (\ref{Eq.DefLambda}) for $G_k(m,n)$ as in (\ref{Eq.LPTk}). In addition, set $\lambda(m,n) = \emptyset$ when $m = 0$ or $n = 0$. Finally, fix any $M, N \geq 0$ and a down-right path $\gamma = (v_0, \dots, v_{M+N})$ from $(N,N)$ to $(M+N,0)$, whose corresponding word is $s_1s_2\cdots s_{M+N}$. Then, for any partitions $\lambda^0, \dots, \lambda^{M+N}$ we have that 
\begin{equation}\label{Eq.HSDistr}
\begin{split}
&\mathbb{P}(\lambda(v_i) = \lambda^i \mbox{ for } i = 0, \dots, M+N) \\
& = Z \cdot \tau_{\lambda^0}(c) \cdot {\bf 1}\left\{ \lambda^{M+N} = \emptyset \right\} \cdot \prod_{i = 1}^{M+N} Q(\lambda^{i-1}, \lambda^{i}),
\end{split}
\end{equation}
where $Q(\lambda^{i-1}, \lambda^i)$ are as in (\ref{Eq.DefQ}) with $y_i = x_i$, $\tau_{\lambda}(c) = c^{\lambda_1 - \lambda_2 + \lambda_3 - \lambda_4 + \cdots}$, and $Z$ is a normalization constant. If we denote $Y_{<}(\gamma) = \{(a,b) \in Y(\gamma): b < a \}$ with $Y(\gamma)$ as in (\ref{Eq.FerresShapeDef}), then 
$$Z = \prod_{i = 1}^N (1- cx_i) \cdot \prod_{(a,b) \in Y_{<}(\gamma)} (1- x_ax_b).$$
\end{theorem}

\begin{figure}[h]
\centering
\begin{tikzpicture}[scale=0.85]
\begin{scope}[shift={(0,3)}, scale = 0.8]
\def\m{6} 
\def\mm{5} 
\def\n{4} 
\def\nn{3} 
\def\k{3} 
\definecolor{Bg}{gray}{1.0}   
\definecolor{C1}{gray}{0.8} 
\definecolor{C2}{gray}{0.5} 
\definecolor{C3}{gray}{0.2} 
\foreach \x/\y in {3/3, 4/3, 5/3, 5/2, 6/2, 6/1} {
\fill[C2] (\x,\y) rectangle ++(1,-1);}
\foreach \i in {1,...,\n}{
  \foreach \j in {1,...,\m}{
    \draw[black] (\j,\i) rectangle ++(1,-1);
  }
} 
\node at (2.5,1.5) {\(1\)};
\node at (3.5,1.5) {\(1\)};
\node at (4.5,1.5) {\(3\)};
\node at (5.5,1.5) {\(1\)};
\node at (6.5,1.5) {\(2\)};
\node at (2.5,2.5) {\(1\)};
\node at (3.5,2.5) {\(2\)};
\node at (4.5,2.5) {\(0\)};
\node at (5.5,2.5) {\(1\)};
\node at (6.5,2.5) {\(2\)};
\node at (2.5,3.5) {\(3\)};
\node at (3.5,3.5) {\(0\)};
\node at (4.5,3.5) {\(2\)};
\node at (5.5,3.5) {\(1\)};
\node at (6.5,3.5) {\(4\)};

\node at (4,-0.5) {\(W\)};
\end{scope}

\begin{scope}[shift={(-0.5,-1.7)}, scale = 0.8]
\def\m{6} 
\def\mm{5} 
\def\n{4} 
\def\nn{3} 
\def\k{3} 

\definecolor{Bg}{gray}{1.0}   
\definecolor{C1}{gray}{0.8} 
\definecolor{C2}{gray}{0.5} 
\definecolor{C3}{gray}{0.2} 
\foreach \i in {2,...,\n}{
  \foreach \j in {2,...,\m}{
    \draw[black] (\j,\i) rectangle ++(1,-1);
  }
} 
\node at (2.5,1.5) {\(1\)};
\node at (3.5,1.5) {\(2\)};
\node at (4.5,1.5) {\(5\)};
\node at (5.5,1.5) {\(6\)};
\node at (6.5,1.5) {\(8\)};
\node at (2.5,2.5) {\(2\)};
\node at (3.5,2.5) {\(4\)};
\node at (4.5,2.5) {\(5\)};
\node at (5.5,2.5) {\(7\)};
\node at (6.5,2.5) {\(10\)};
\node at (2.5,3.5) {\(5\)};
\node at (3.5,3.5) {\(5\)};
\node at (4.5,3.5) {\(7\)};
\node at (5.5,3.5) {\(8\)};
\node at (6.5,3.5) {\(14\)};

\node at (1.5,2.5) {\(G_1\)};
\end{scope}

\begin{scope}[shift={(9.5,-1.7)}, scale = 0.8]
\def\m{6} 
\def\mm{5} 
\def\n{4} 
\def\nn{3} 
\def\k{3} 

\definecolor{Bg}{gray}{1.0}   
\definecolor{C1}{gray}{0.8} 
\definecolor{C2}{gray}{0.5} 
\definecolor{C3}{gray}{0.2} 
\foreach \i in {2,...,\n}{
  \foreach \j in {2,...,\m}{
    \draw[black] (\j,\i) rectangle ++(1,-1);
  }
} 
\node at (2.5,1.5) {\(1\)};
\node at (3.5,1.5) {\(2\)};
\node at (4.5,1.5) {\(5\)};
\node at (5.5,1.5) {\(6\)};
\node at (6.5,1.5) {\(8\)};
\node at (2.5,2.5) {\(2\)};
\node at (3.5,2.5) {\(5\)};
\node at (4.5,2.5) {\(8\)};
\node at (5.5,2.5) {\(10\)};
\node at (6.5,2.5) {\(14\)};
\node at (2.5,3.5) {\(5\)};
\node at (3.5,3.5) {\(8\)};
\node at (4.5,3.5) {\(13\)};
\node at (5.5,3.5) {\(16\)};
\node at (6.5,3.5) {\(24\)};

\node at (1.5,2.5) {\(G_3\)};
\end{scope}

\begin{scope}[shift={(4.5,-1.7)}, scale = 0.8]
\def\m{6} 
\def\mm{5} 
\def\n{4} 
\def\nn{3} 
\def\k{3} 

\definecolor{Bg}{gray}{1.0}   
\definecolor{C1}{gray}{0.8} 
\definecolor{C2}{gray}{0.5} 
\definecolor{C3}{gray}{0.2} 
\foreach \i in {2,...,\n}{
  \foreach \j in {2,...,\m}{
    \draw[black] (\j,\i) rectangle ++(1,-1);
  }
} 
\node at (2.5,1.5) {\(1\)};
\node at (3.5,1.5) {\(2\)};
\node at (4.5,1.5) {\(5\)};
\node at (5.5,1.5) {\(6\)};
\node at (6.5,1.5) {\(8\)};
\node at (2.5,2.5) {\(2\)};
\node at (3.5,2.5) {\(5\)};
\node at (4.5,2.5) {\(8\)};
\node at (5.5,2.5) {\(10\)};
\node at (6.5,2.5) {\(14\)};
\node at (2.5,3.5) {\(5\)};
\node at (3.5,3.5) {\(8\)};
\node at (4.5,3.5) {\(11\)};
\node at (5.5,3.5) {\(14\)};
\node at (6.5,3.5) {\(21\)};

\node at (1.5,2.5) {\(G_2\)};
\end{scope}

\begin{scope}[shift={(6.5,0)},scale = 2]
\def\m{6} 
\def\mm{5} 
\def\n{4} 
\def\nn{3} 
\def\k{3} 

\draw[line width=1pt]
     (1,3) -- (2,3) -- (3,3) -- (3,2) -- (4,2) -- (4,1);
 
\def\r{0.05}
\draw[fill = black] (1,3) circle [radius=\r];
\draw[fill = black] (2,3) circle [radius=\r];
\draw[fill = black] (3,3) circle [radius=\r];
\draw[fill = black] (3,2) circle [radius=\r];
\draw[fill = black] (4,2) circle [radius=\r];
\draw[fill = black] (4,1) circle [radius=\r];

\node at (1.5,3.18) {\(\preceq\)};
\node at (2.5,3.18) {\(\preceq\)};
\node at (3.15,2.5) {\rotatebox{-90}{$\succeq$}};
\node at (3.6,2.18) {\(\preceq\)};
\node at (4.15,1.5) {\rotatebox{-90}{$\succeq$}};

\node at (1.5,2.82) {\(R\)};
\node at (2.5,2.82) {\(R\)};
\node at (2.85,2.5) {$D$};
\node at (3.6,1.82) {\(R\)};
\node at (3.85,1.5) {$D$};

\node at (1,3.18) {\((4,1)\)};
\node at (2,3.18) {\((5,3)\)};
\node at (3,3.18) {\((7,3)\)};
\node at (3.2,2.18) {\((6)\)};
\node at (4,2.18) {\((8)\)};
\node at (4.15,1.18) {\(\emptyset\)};

\end{scope}

\end{tikzpicture}  
\caption{The figure depicts a sample symmetric random environment $W$, a down-right path $\gamma$ with corresponding word $RRDRD$, and the higher-rank last passage times $G_k(m,n)$ from (\ref{Eq.LPTk}). It also shows the partitions $\{\lambda(v)\}_{v \in \gamma}$ from Theorem \ref{T2}, as well as the interlacing relations for consecutive partitions along the path $\gamma$ implied by (\ref{Eq.DefSchur}), (\ref{Eq.DefQ}), and (\ref{Eq.HSDistr}).}
\label{Fig.T2}
\end{figure}

Theorem \ref{T2} generalizes \cite[Proposition 3.10]{BBCS}, which establishes the distributional equality between $(\lambda_1(v))_{v \in \gamma}$ (or equivalently $(G_1(v))_{v \in \gamma}$) and the Pfaffian Schur process. The latter was also established by similar means in \cite[Section 4.2.1]{BBNV18}. It is worth mentioning that for the special case when the path $\gamma$ goes straight down, \cite[Remark 3.14]{BBCS} sketches an argument for proving Theorem \ref{T2}. Our proof formalizes that outline and generalizes it to arbitrary down-right paths. Using a version of Greene's theorem, see Corollary \ref{P.GreeneHS}, one can show that Theorem \ref{T2} is equivalent to \cite[Proposition 4.4]{BBNV18}, which was stated without proof. 

The proofs of Theorems \ref{T1} and \ref{T2} are given in Section \ref{Section4} and are based on a straightforward inductive argument. At a conceptual level, however, this argument can be interpreted as applying a certain Markovian dynamics on sequences of interlacing partitions. This dynamics was used in \cite{betea2018perfect, BBNV18} and is of what is called {\em RSK-type}. It is different from the {\em push-block} dynamics that was used to prove \cite[Proposition 3.10]{BBCS}. We refer the interested reader to \cite{MP17} for background on RSK-type and push-block dynamics and their generalizations. 

%
%
\subsection{Discussion of Theorems \ref{T1} and \ref{T2}}\label{Section2.4} Theorems \ref{T1} and \ref{T2} lie at the intersection of several well-developed areas of probability, combinatorics, and representation theory. The purpose of this section is to briefly review some of the relevant background and place these results within the existing literature.\\

{\bf \raggedleft Background on RSK.} At a high level, the RSK algorithm is a bijection between integer matrices with non-negative entries and pairs of interlacing sequences of partitions (equivalently, semistandard Young tableaux). Textbook treatments of RSK can be found in \cite{Fulton, Sagan, Stanley99}. Over the years, the correspondence has been understood from many different perspectives, some of which we briefly describe below.\\

{\em \raggedleft Insertion/bumping.} This is the classical formulation due to Robinson, Schensted, and Knuth \cite{K70, Schensted61}. Modern treatments can be found in \cite{Fulton, Sagan, Stanley99}.\\

{\em \raggedleft Matrix-ball construction.} This is a generalization of Viennot's geometric construction \cite{Viennot77}; a modern exposition can be found in \cite{Fulton}.\\

{\em \raggedleft Growth diagrams.} This formulation is based on Fomin's local rules \cite{F86,F95a,F95b}. A modern treatment can be found in \cite{K06}.\\

{\em \raggedleft Piecewise-linear RSK.} This is the $(\max,+)$ formulation of Berenstein and Kirillov \cite{BK95}. A modern treatment can be found in \cite{Z22}.\\

{\em \raggedleft Watermelon maps.} This formulation is based on ideas that go back to Pitman \cite{Pitman75} and is explained in \cite{DNV22}.\\

The above list is necessarily incomplete, as RSK is probably the most studied bijection in algebraic combinatorics and the literature surrounding it is vast. Even within each of the categories above, we have omitted many important contributions. A particularly accessible overview of several of these perspectives can be found in \cite{Z22}. The paper also discusses the {\em geometric lifting} of RSK, which plays a central role in the study of random polymers \cite{COSZ, OSZ14}.

In the present paper, we adopt the formulation based on Fomin's growth diagrams. In particular, the required background on RSK is essentially contained in \cite[Section 4.1]{K06}, whose contents we recall in Section \ref{Section3}. \\

{\bf \raggedleft Schur processes.} The nearest-neighbor product structure of the measures on the right-hand sides of (\ref{Eq.FSDistr}) and (\ref{Eq.HSDistr}) shows that these laws are one-dimensional Gibbs measures.  The corresponding laws are special instances of the {\em Schur processes} of \cite{OR03} in the full-space setting and of the {\em Pfaffian Schur processes} of \cite{BR05} in the half-space setting. 

More specifically, let $\mathbb{S}$ be the distribution on sequences of partitions $(\lambda, \mu)$ satisfying
\begin{equation}\label{Eq.SeqPart}
\emptyset \subset \lambda^{(1)} \supset \mu^{(1)} \subset \lambda^{(2)} \supset \mu^{(2)} \subset \cdots \supset \mu^{(M+ N)} \subset \lambda^{(M + N +1)} \supset \emptyset,
\end{equation}
as in \cite[Definition 3]{B11}, with $N$ replaced with $M + N + 1$, $\rho_0^+ = 0 =\rho^-_{M+N+1}$, and for $i = 1, \dots, M+N$  
\begin{equation}\label{Eq.Specializations}
\begin{split}
&\rho_i^+ = \begin{cases} x_t & \mbox{ if } v_{i-1} \in \{t-1\} \times \mathbb{Z}  \mbox{ and $s_i = R$},\\ 0 & \mbox{ otherwise}; \end{cases} \\
& \rho^-_{i} =  \begin{cases} y_t & \mbox{ if } v_{i-1} \in \mathbb{Z} \times \{t\} \mbox{ and $s_i = D$}, \\ 0 & \mbox{ otherwise}. \end{cases}
\end{split}
\end{equation}
All of the above specializations $\rho^{\pm}_i$ are in a single variable. Then, the right-hand side of (\ref{Eq.FSDistr}) is precisely
$$\mathbb{S} \left( \lambda^{(i)} = \lambda^{i-1} \mbox{ for } i = 1, \dots, M + N+1 \right).$$

Similarly, let $\mathbb{S}^{\mathrm{hs}}$ be the distribution defined in \cite[Section 3]{BR05} on sequences of partitions $(\lambda, \mu)$ satisfying (\ref{Eq.SeqPart}), with $n = M + N + 1$, $\rho_0^+ = c$, $\rho^-_{M+N+1} = 0$, and $\rho^{\pm}_i$ for $i = 1, \dots, M+N$ as in (\ref{Eq.Specializations}) with $y_t = x_t$. Then, the right-hand side of (\ref{Eq.HSDistr}) is precisely
$$\mathbb{S}^{\mathrm{hs}} \left( \lambda^{(i)} = \lambda^{i-1} \mbox{ for } i = 1, \dots, M + N+1 \right).$$

Using standard properties of Schur symmetric functions, one can derive explicit formulas for the marginals of the above processes. We refer the reader to \cite[Chapter I]{Mac} for a comprehensive treatment and to \cite{BorodinGorin2016} for a more accessible introduction to symmetric functions. In particular, combining the {\em branching relations} and the {\em skew-Cauchy identity} for Schur functions, one obtains that the partition $\lambda(m,n)$ from Theorem \ref{T1} has distribution
\begin{equation}\label{Eq.SchurMeasure}
\mathbb{P}(\lambda(m,n) = \lambda) = \prod_{i = 1}^m \prod_{j = 1}^n \frac{1}{1-x_i y_j} \cdot s_{\lambda}(x_1, \dots, x_m) s_{\lambda}(y_1, \dots, y_n).
\end{equation}
Here $s_{\lambda}$ denotes the Schur symmetric polynomial indexed by $\lambda$, and the measure in (\ref{Eq.SchurMeasure}) is a special case of the {\em Schur measures} introduced in \cite{Okounkov2001}. An important consequence of the symmetry of Schur polynomials is that the law of $\lambda(m,n)$ is invariant under arbitrary permutations of the variables $x_1, \dots, x_m$ and, independently, of the variables $y_1, \dots, y_n$. While this symmetry is far from apparent from the original last-passage percolation formulation, it becomes immediate from the Schur-process perspective.\\

{\bf \raggedleft Asymptotic analysis.} The distributional identities in Theorems \ref{T1} and \ref{T2} provide a natural starting point for the asymptotic analysis of the partitions $\lambda(m,n)$ as $m,n \rightarrow \infty$. 

Firstly, they imply that $\{\lambda_{i}(v): i \geq 1, v \in \gamma \}$ form a {\em determinantal point process} in the full-space setting and a {\em Pfaffian point process} in the half-space setting. This follows from the connection to Schur processes and Pfaffian Schur processes described above, together with the facts that Schur processes are determinantal, as shown in \cite{OR03}, whereas Pfaffian Schur processes are Pfaffian, as shown in \cite{BR05}. For background on determinantal point processes, we refer the reader to \cite{BDPP11, HKP, Sosh00}, and for background on Pfaffian point processes to \cite{R00} and \cite[Appendix B]{OQR17}. 

From the perspective of asymptotic analysis, the key advantage of having a determinantal or Pfaffian point process structure is the availability of {\em exact formulas} for many observables of the model in terms of a {\em correlation kernel}. For Schur processes, these kernels admit a variety of representations that can be analyzed using methods from asymptotic analysis, including orthogonal-polynomial techniques, Riemann--Hilbert methods, and the method of steepest descent. A careful asymptotic analysis of the resulting kernels then leads to finite-dimensional convergence statements. Early works utilizing this approach go back to \cite{J00,J03} for full-space LPP and to \cite{BR01c, FR07, SI04} for half-space LPP.

The second consequence of the distributional identities in Theorems \ref{T1} and \ref{T2} is that they provide an interpretation of $\lambda(m,n)$ as a {\em Gibbsian line ensemble}. To keep the discussion simple, we focus on Theorem \ref{T1} when the word corresponding to $\gamma$ is
\begin{equation}\label{Eq.WordIntro}
s_1 \dots s_{M+N} = \underbrace{R, \dots, R}_\text{$M$ times}\underbrace{D, \dots, D}_\text{$N$ times}.
\end{equation}
Enumerating the points along $\gamma$ with $v_t$ for $t = 0, \dots, M+N$, one can plot the points $(t, \lambda_i(v_t))$ in $\mathbb{Z}^2$ and linearly interpolate them to produce a sequence of random continuous curves, see Figure \ref{Fig.LineEnsemble}.

\begin{figure}[ht]
  \begin{center}
    \begin{tikzpicture}[scale=0.7]
    \begin{scope}
        \def\r{0.1}
    \draw[dotted, gray] (0,0) grid (8,6);

    \draw[line width=2pt, gray] (5,0)--(5,6);
        \draw[fill = black] (0,0) circle [radius=\r];
        \draw[fill = black] (1,4) circle [radius=\r];
        \draw[fill = black] (2,5) circle [radius=\r];
        \draw[fill = black] (3,5) circle [radius=\r];
        \draw[fill = black] (4,5) circle [radius=\r];
        \draw[fill = black] (5,5) circle [radius=\r];
        \draw[fill = black] (6,4) circle [radius=\r];
        \draw[fill = black] (7,2) circle [radius=\r];
        \draw[fill = black] (8,0) circle [radius=\r];
        \draw[-][black] (0,0) -- (1,4);
        \draw[-][black] (1,4) -- (2,5);
        \draw[-][black] (2,5) -- (3,5);
        \draw[-][black] (3,5) -- (4,5);
        \draw[-][black] (4,5) -- (5,5);
        \draw[-][black] (5,5) -- (6,4);
        \draw[-][black] (6,4) -- (7,2);
        \draw[-][black] (7,2) -- (8,0);

        \draw[fill = black] (0,0) circle [radius=\r];
        \draw[fill = black] (1,0) circle [radius=\r];
        \draw[fill = black] (2,2) circle [radius=\r];
        \draw[fill = black] (3,3) circle [radius=\r];
        \draw[fill = black] (4,3) circle [radius=\r];
        \draw[fill = black] (5,3) circle [radius=\r];
        \draw[fill = black] (6,3) circle [radius=\r];
        \draw[fill = black] (7,0) circle [radius=\r];
        \draw[fill = black] (8,0) circle [radius=\r];
        \draw[-][black] (0,0) -- (1,0);
        \draw[-][black] (1,0) -- (2,2);
        \draw[-][black] (2,2) -- (3,3);
        \draw[-][black] (3,3) -- (4,3);
        \draw[-][black] (4,3) -- (5,3);
        \draw[-][black] (5,3) -- (6,3);
        \draw[-][black] (6,3) -- (7,0);
        \draw[-][black] (7,0) -- (8,0);

        \draw[fill = black] (0,0) circle [radius=\r];
        \draw[fill = black] (1,0) circle [radius=\r];
        \draw[fill = black] (2,0) circle [radius=\r];
        \draw[fill = black] (3,1) circle [radius=\r];
        \draw[fill = black] (4,2) circle [radius=\r];
        \draw[fill = black] (5,2) circle [radius=\r];
        \draw[fill = black] (6,0) circle [radius=\r];
        \draw[fill = black] (7,0) circle [radius=\r];
        \draw[fill = black] (8,0) circle [radius=\r];
        \draw[-][black] (0,0) -- (1,0);
        \draw[-][black] (1,0) -- (2,0);
        \draw[-][black] (2,0) -- (3,1);
        \draw[-][black] (3,1) -- (4,2);
        \draw[-][black] (4,2) -- (5,2);
        \draw[-][black] (5,2) -- (6,0);
        \draw[-][black] (6,0) -- (7,0);
        \draw[-][black] (7,0) -- (8,0);

        \draw (0, -0.5) node{$0$};
        \draw (5, -0.5) node{$M$};
        \draw (8, -0.5) node{$M+N$};
        \draw (4.5, 5.3) node{$\lambda_1$};
        \draw (4.5, 3.3) node{$\lambda_2$};
        \draw (4.5, 2.3) node{$\lambda_3$};

    \end{scope}

    \end{tikzpicture}
  \end{center}
  \caption{The figure depicts the top three curves in $\{\lambda_i(v): i \geq 1, v \in \gamma\}$ from Theorem \ref{T1} for a down-right path $\gamma$ with word as in (\ref{Eq.WordIntro}).}
  \label{Fig.LineEnsemble}
\end{figure}

As shown in \cite{D24b}, by restricting the curves to either the interval $[0,M]$ or $[M, M+N]$, the curves in $\{\lambda_i(v): i \geq 1, v \in \gamma\}$ can be viewed as a sequence of independent geometric random walk bridges, which have been conditioned to interlace. The latter description allows one to make a strong comparison between the curves of this ensemble and avoiding Brownian bridges, allowing one to obtain estimates on the modulus of continuity of these curves. When combined with the finite-dimensional convergence discussed above, this leads to enhanced forms of convergence including functional limit theorems in the spirit of Donsker's theorem. These ideas have received considerable attention over the last several years leading to many strong asymptotic results. In the full-space setting a combination of exact formulas and Gibbsian ensembles have been central to the works \cite{DNV23, D24b, ED24a}. In the half-space setting they have been utilized in \cite{DDY26, DY25, DZ25, Zhou25}. 

We finally mention that the down-right nature of the path $\gamma$ in Theorems \ref{T1} and \ref{T2} is crucial for the aforementioned works. Such paths are called {\em space-like}, and if one considers other up-right paths instead, called {\em time-like}, then one loses the determinantal/Pfaffian point process structure. Nevertheless, over the last several years, there has been substantial progress in understanding LPP along time-like paths starting with the seminal work \cite{J19}. \\

{\bf \raggedleft Limit transitions and alternative domains.} Theorem \ref{T1} concerns geometric last passage percolation in a quadrant. Several other integrable models can be obtained from this framework through suitable degenerations or modifications of the underlying geometry.\\

{\em \raggedleft Exponential LPP.} Fix parameters $\alpha_i, \beta_j > 0$ and set $x_i = e^{-\varepsilon \alpha_i}$ and $y_j = e^{-\varepsilon \beta_j}$. Then, as $\varepsilon \rightarrow 0+$, the random variables $\varepsilon w_{i,j}$ converge weakly to independent exponential random variables with rates $\alpha_i + \beta_j$. Moreover, the rescaled last passage times $\varepsilon G_k(m,n)$ converge weakly to the corresponding last passage times in the exponential model defined by the same formula \eqref{Eq.LPTk} with the geometric weights of $W$ replaced by the limiting exponential weights.\\

{\em \raggedleft Poisson LPP.} Fix two continuous functions $f, g: [0,\infty) \rightarrow [0,\infty)$ and set $x_i = \varepsilon f(\varepsilon i)$, $y_i = \varepsilon g(\varepsilon i)$. Then, as $\varepsilon \rightarrow 0+$, the last passage times $G_k( \lfloor x \varepsilon^{-1} \rfloor ,\lfloor y \varepsilon^{-1} \rfloor)$ converge weakly and jointly over $(x,y) \in (0, \infty)^2$. When $f = g = 1$, the resulting model is related to the Poissonization of the uniform measure on permutations, and the last passage times correspond to the lengths of the longest increasing subsequences. The last statement is due to \cite{BDJ99, Joh01}, see also \cite{BorodinGorin2016, Romik15}.\\

{\em \raggedleft Brownian LPP.} If we fix $x_i = y_j = q \in (0,1)$, then as $N \rightarrow \infty$, the scaled last passage times 
$$\left\{\frac{G_k( \lfloor Nt \rfloor,n) - k\mu Nt }{\sigma \sqrt{N} } \right\}_{k = 1}^n, \mbox{ with } \mu = \frac{q}{1-q}, \mbox{ and } \sigma^2 = \frac{q}{(1-q)^2},$$
converge weakly to continuous processes on $[0,\infty)$ that are related to Dyson Brownian motion.\\

For special choices of parameters, the above limit transitions are discussed in \cite{DNV23}. The latter paper also explains the connection to avoiding Bernoulli random walks. As suggested by the anonymous referee, we explain the limit transition to the exponential LPP model, and derive analogues of Theorems \ref{T1} and \ref{T2} in Appendix \ref{SectionB}.\\

A large body of work has also considered last passage percolation in domains that are different from the first quadrant. For example, \cite{BBNV18} considers LPP on a strip with various types of boundary conditions including open and periodic. In order to maintain the connection to Schur processes, one needs to take special boundary conditions which involve specializing certain parameters into geometric progressions. For recent progress on the study of periodic LPP on a strip (i.e. with cylindrical geometry), we refer to \cite{IMS23} and the preprint \cite{BO21}. More recently, \cite{AGL24} has considered LPP with weights being set to zero outside a lower triangular domain, see also \cite{BNN25}. An interesting feature of the latter works is that the Schur functions that arise from the RSK correspondence get replaced by {\em Demazure characters} and {\em atoms}, which are non-symmetric polynomials.

%
%
\section{RSK fillings and Greene's theorem}\label{Section3} In this section we recall the generalizations of the RSK correspondence and Greene's theorem as defined by Krattenthaler in \cite[Section 4.1]{K06}, which are based on Fomin's growth diagrams \cite{F86,F95a,F95b}. In the general setup, described in Section \ref{Section3.1}, the RSK is a one-to-one correspondence between integer fillings of Ferrers shapes and sequences of interlacing partitions, and Greene's theorem provides a global description of this bijection in terms of NE-chains. We also explain how these results specialize to the setting of symmetric fillings of symmetric Ferrers shapes in Section \ref{Section3.2}.

%
%
\subsection{General setup}\label{Section3.1} We recall that partitions, interlacing and down-right paths were defined in Section \ref{Section2.3}. We start with the following algorithms, which for given partitions $\mu, \nu$ define the forward and backward maps of a bijection between pairs $(\rho, m)$ and $\lambda$. Here, $\rho$ is a partition such that $\rho \preceq \mu, \rho \preceq \nu$ and $m$ is a non-negative integer, and $\lambda$ is a partition such that $\lambda \succeq \mu, \lambda \succeq \nu$. 

\vspace{2mm}
\begin{minipage}{0.95\textwidth}
{\bf \raggedleft Algorithm F$^1$} in \cite[Section 4.1]{K06}.
\begin{enumerate}
    \item[Step 0.] Set $\carry := m$ and $i:= 1$.
    \item[Step 1.] Set $\lambda_i := \max(\mu_i, \nu_i) + \carry$.
    \item[Step 2.] If $\lambda_i = 0$, then stop. The output is $\lambda = (\lambda_1, \lambda_2, \dots, \lambda_{i-1}, 0, \dots)$. If not, then set $\carry:= \min(\mu_i, \nu_i) - \rho_i$ and $i :=i+1$. Go to Step 1. 
\end{enumerate}
\end{minipage}

\vspace{2mm}
\begin{minipage}{0.95\textwidth}
{\bf \raggedleft Algorithm B$^1$} in \cite[Section 4.1]{K06}.
\begin{enumerate}
    \item[Step 0.] Set $i:= \max\{j: \lambda_j > 0\}$ and $\carry:= 0$.
    \item[Step 1.] Set $\rho_i := \min(\mu_i, \nu_i) - \carry$.
    \item[Step 2.] Set $\carry:=\lambda_i - \max(\mu_i, \nu_i)$ and $i:=i-1$. If $i = 0$, then stop. The output of the algorithm is $\rho = (\rho_1, \rho_2, \dots)$ and $m = \carry$. If not, go to Step 1.
\end{enumerate}
\end{minipage}
\vspace{2mm}

\begin{remark} Observe that for the above algorithms we have
$$\lambda_1 \hspace{-0.5mm} = \hspace{-0.5mm} \max(\mu_1, \nu_1) + m, \mbox{ and } \lambda_i \hspace{-0.5mm} = \hspace{-0.5mm} \max(\mu_i, \nu_i) + \min(\mu_{i-1}, \nu_{i-1}) - \rho_{i-1} \mbox{ for } i \geq 2.$$
Summing the latter over $i$ gives
\begin{equation}\label{Eq.MassPres}
\sum_{i \geq 1} (\lambda_i + \rho_i) = m + \sum_{i \geq 1} (\mu_i + \nu_i).
\end{equation} 
In addition, if $\mu = \nu$, we have
\begin{equation}\label{Eq.MassPres2}
\sum_{i \geq 1} (\lambda_{2i-1} - \lambda_{2i}) = (m + \mu_1) - (\mu_2 + \mu_1 - \rho_1) + \cdots = m + \sum_{i \geq 1} (\rho_{2i - 1} - \rho_{2i}).
\end{equation} 
\end{remark}

Suppose $m, n \geq 0$, and we are given a down-right path $\gamma = (v_0, \dots, v_{m+n})$, where $v_0 = (0,n)$ and $v_{m+n} = (m,0)$. Recall that any such path is encoded by a word $s_1s_2\cdots s_{m+n}$ of $n$ letters $D$ and $m$ letters $R$. Recall also the set $Y(\gamma)$ from (\ref{Eq.FerresShapeDef}), which one can visualize as a Ferrers shape in $\mathbb{Z}_{\geq 1}^2$, see Figure \ref{Fig.DownRightPath}. 

A {\em filling} of $Y(\gamma)$ is a map $f: Y(\gamma) \rightarrow \mathbb{Z}_{\geq 0}$. Given a filling $f$ of $Y(\gamma)$, we can use Algorithm F$^1$ to construct partitions $\{\lambda^{(u,v)}: (u,v) \in Y(\gamma)\}$ as follows. We take any sequence of down-right paths $\gamma_0, \gamma_1, \dots, \gamma_r$ connecting $(0,n)$ to $(m,0)$, such that $Y(\gamma_0) = \emptyset$, $\gamma_r = \gamma$, and 
$$Y(\gamma_i) \subset Y(\gamma_{i+1}), \hspace{3mm} |Y(\gamma_{i+1}) \setminus Y(\gamma_{i})| = 1.$$
In words, we are taking a sequence of paths whose enclosed Ferrers shapes are growing by one box at each step. If $(m_i,n_i)$ is the vertex in $Y(\gamma_i)$ that is not in $Y(\gamma_{i-1})$, we let $\lambda^{(m_i,n_i)}$ be the output of applying Algorithm F$^1$ to $\rho = \lambda^{(m_i-1, n_i-1)}$, $\nu = \lambda^{(m_i-1, n_i)}$, $\mu = \lambda^{(m_i, n_i-1)}$, and $m = f(m_i, n_i)$. This construction is initiated by setting $\lambda^{(m,n)} = \emptyset$ if $m = 0$ or $n = 0$. It is not hard to see from the local nature of Algorithm F$^1$ that the obtained $\{\lambda^{(u,v)}: (u,v) \in Y(\gamma)\}$ does not depend on the sequence of down-right paths $\gamma_0, \gamma_1, \dots, \gamma_r$ as long as they satisfy the above conditions. We denote the sequence $(\lambda^{v_0}, \dots, \lambda^{v_{m+n}})$ by $\rsk_{\gamma}(f)$. 

With the above notation we can state the RSK correspondence from \cite[Theorem 7]{K06}.
\begin{proposition}\label{P.RSKFS} The map $\rsk_{\gamma}$ defines a bijection between fillings of $Y(\gamma)$ and sequences of partitions $(\emptyset = \lambda^{0}, \dots, \lambda^{m+n} = \emptyset)$, such that $\lambda^{i} \succeq \lambda^{i-1}$ if $s_i = R$, whereas $\lambda^{i} \preceq \lambda^{i-1}$ if $s_i = D$.  
\end{proposition}

We next state the following generalization of Greene's theorem from \cite[Theorem 8]{K06}.
\begin{proposition}\label{P.GreeneFS} For a filling $f$ of $Y(\gamma)$, define $\{\mu^{(u,v)}_{i}: i \geq 1, (u,v) \in Y(\gamma) \}$ through the equations
\begin{equation}\label{Eq.GrFS}
\sum_{i = 1}^k \mu^{(u,v)}_i = h_k(f[u,v]), \mbox{ where } f[u,v] = (f(i,j): 1 \leq i \leq u, 1 \leq j \leq v),
\end{equation}
and $h_k$ are as in (\ref{Eq.HDef}). Then, $(\mu^{(u,v)}_i: i \geq 1) = \lambda^{(u,v)}$ for all $(u,v) \in Y(\gamma)$. In particular, in view of Proposition \ref{P.RSKFS}, (\ref{Eq.GrFS}) defines a bijection between fillings of $Y(\gamma)$ and sequences of partitions $(\emptyset = \mu^{v_0}, \dots, \mu^{v_{m+n}} = \emptyset)$, such that $\mu^{v_i} \succeq \mu^{v_{i-1}}$ if $s_i = R$, whereas $\mu^{v_i} \preceq \mu^{v_{i-1}}$ if $s_i = D$. 
\end{proposition}
\begin{remark}\label{R.S31}
Proposition \ref{P.GreeneFS} appears with $h_k$ replaced by $g_k$ as in (\ref{Eq.GDef}) and (\ref{Eq.GDef2}) in \cite[Section 8.2]{DNV22} for the special case when 
$$s_1 \dots s_{m+n} = \underbrace{R, \dots, R}_\text{$m$ times} \underbrace{D, \dots, D}_\text{$n$ times},$$ 
and for general down-right paths in \cite[Theorem 5]{BCGR22}. In both papers the result is attributed to \cite{K06}, and the distinction between $h_k$ and $g_k$, that is, between NE-chains and up-right paths, appears to be overlooked. Of course, Proposition \ref{P.GreeneFS} holds with $h_k$ replaced by $g_k$ in view of Lemma \ref{L.GequalsH}.
\end{remark}

We end this section with the proof of Lemma \ref{L.HRSK}.

\begin{proof}[Proof of Lemma \ref{L.HRSK}] Assume first that $a_{i,j} \in \mathbb{Z}_{\geq 0}$. Setting $\mu_k = h_k(A) - h_{k-1}(A)$ for $k \geq 1$, we have from Proposition \ref{P.GreeneFS} that $(\mu_k: k \geq 1)$ is a partition, which proves (\ref{Eq.HDec}). If $a_{i,j} \in \mathbb{Q}_{\geq 0}$, then we can clear the denominators of the entries and deduce (\ref{Eq.HDec}) from the integer case. Lastly, we can deduce (\ref{Eq.HDec}) when $a_{i,j} \in \mathbb{R}_{\geq 0}$ by taking limits over $\mathbb{Q}_{\geq 0}$.
\end{proof}

%
%
\subsection{Symmetric setup}\label{Section3.2} In this section we restrict the map $\rsk_{\gamma}$ from Section \ref{Section3.1} to symmetric fillings and deduce half-space analogues of Propositions \ref{P.RSKFS} and \ref{P.GreeneFS}. In the sequel we continue with the same notation as in Section \ref{Section3.1}.

We say that a down-right path $\gamma$ is {\em symmetric} if $(i,j) \in \gamma$ implies $(j,i) \in \gamma$. Note that any symmetric down-right path $\gamma$ contains an odd number of vertices $2r+1$, and that if $s_1s_2 \cdots s_{2r}$ is the word corresponding to $\gamma$ we have $s_i = D$ if and only if $s_{2r-i +1} = R$ for $i = 1, \dots, 2r$. In addition, note that $(i,j) \in Y(\gamma)$ if and only if $(j,i) \in Y(\gamma)$. We say that a filling $f$ of $Y(\gamma)$ is {\em symmetric} if $f(i,j) = f(j,i)$ for all $(i,j) \in Y(\gamma)$. 

Suppose that $m,n \geq 0$, and we are given a symmetric down-right path $\gamma = (v_{0}, \dots, v_{2m + 2n})$, where $v_{m+n} = (n,n)$ and $v_{2m + 2n} = (m+ n, 0)$. By a straightforward inductive argument on $|Y(\gamma)|$ one readily checks that applying successively Algorithm F$^1$ to a symmetric filling of $Y(\gamma)$ results in a sequence $\rsk_{\gamma}(f) = (\lambda^{v_0}, \dots, \lambda^{v_{2m+2n}})$ such that $\lambda^{v_{i}} = \lambda^{v_{2m + 2n - i}}$ for $i = 0, \dots, 2m + 2n$. Of course, by Proposition \ref{P.RSKFS} this sequence further satisfies 
$$\lambda^{v_i} \succeq \lambda^{v_{i-1}} \mbox{ if $s_i = R$, whereas $\lambda^{v_i} \preceq \lambda^{v_{i-1}}$ if $s_i = D$.}$$ 
Conversely, if we start from a sequence $\lambda^{v_0}, \dots, \lambda^{v_{2m+2n}}$ that satisfies the above two conditions, and apply successively Algorithm B$^1$, we directly verify that we obtain a symmetric filling of $Y(\gamma)$. Combining the above observations, and setting $\rskhs_{\gamma} = (\lambda^{v_{m+n}}, \dots, \lambda^{v_{2m+2n}})$, we obtain the following corollaries to Propositions \ref{P.RSKFS} and \ref{P.GreeneFS}.

\begin{corollary}\label{C.RSKHS} Assume the same notation as in this section. The map $\rskhs_{\gamma}$ defines a bijection between symmetric fillings of $Y(\gamma)$ and sequences of partitions $(\lambda^{0}, \dots, \lambda^{m+n} = \emptyset)$, such that $\lambda^{i} \succeq \lambda^{i-1}$ if $s_{m+n+i} = R$, whereas $\lambda^{i} \preceq \lambda^{i-1}$ if $s_{m+n+i} = D$.  
\end{corollary}

\begin{corollary}\label{P.GreeneHS} Assume the same notation as in this section. Then, (\ref{Eq.GrFS}) defines a bijection between symmetric fillings of $Y(\gamma)$ and sequences of partitions $(\emptyset = \mu^{v_0}, \dots, \mu^{v_{2m+2n}} = \emptyset)$, such that (1) $\mu^{v_{i}} = \mu^{v_{2m+2n - i}}$ for $i = 0, \dots, 2m + 2n$ and (2) $\mu^{v_i} \succeq \mu^{v_{i-1}}$ if $s_i = R$, whereas $\mu^{v_i} \preceq \mu^{v_{i-1}}$ if $s_i = D$.
\end{corollary}

%
%
\section{Proofs of Theorems \ref{T1} and \ref{T2}}\label{Section4} In Section \ref{Section4.1} we give the proof of Theorem \ref{T1}, and in Section \ref{Section4.2} we give the one for Theorem \ref{T2}. We continue with the notation from Sections \ref{Section2} and \ref{Section3}.

%
%
\subsection{Proof of Theorem \ref{T1}}\label{Section4.1} We proceed to prove (\ref{Eq.FSDistr}) by induction on $|Y(\gamma)|$. For the base case $|Y(\gamma)| = 0$, we have that the word corresponding to $\gamma$ is 
\begin{equation}\label{Eq.Extreme}
s_1 \dots s_{M+N} = \underbrace{D, \dots, D}_\text{$N$ times}\underbrace{R, \dots, R}_\text{$M$ times} .
\end{equation}
By Proposition \ref{P.GreeneFS} we have $\mathbb{P}(\lambda(v_i) = \emptyset \mbox{ for } i = 0, \dots, M+N) = 1$, which implies (\ref{Eq.FSDistr}). 

We now fix $k \geq 1$, suppose that (\ref{Eq.FSDistr}) holds when $|Y(\gamma)| \leq k-1$, and proceed to prove it when $|Y(\gamma)| = k$. Since $|Y(\gamma)| = k \geq 1$, we know that the word corresponding to $\gamma$ is not as in (\ref{Eq.Extreme}), and so must contain $s_{i_0}s_{i_0+1} = RD$, for some $i_0 \in \{1, \dots, M+N-1\}$. Let $\gamma'$ be the down-right path whose word is the same as $\gamma$, but with $s_{i_0}s_{i_0+1} = DR$ instead, and note that $Y(\gamma) = Y(\gamma') \cup \{(a_0,b_0)\}$ for some $(a_0,b_0)$, see the right side of Figure \ref{Fig.DownRightPath}. In addition, if $\gamma' = (v_0', \dots, v_{M+N}')$, we have 
\begin{equation}\label{Eq.PT1E1}
\begin{split}
&v_{i_0} = (a_0,b_0),\hspace{2mm} v_{i_0-1} = (a_0-1,b_0),\hspace{2mm} v_{i_0+1} = (a_0, b_0-1), \\
&v_i' = v_i \mbox{ for } i \neq i_0, \hspace{2mm} v_{i_0}' = (a_0-1, b_0-1).
\end{split}
\end{equation}

We now have by countable additivity and throwing away terms of zero probability
\begin{equation}\label{Eq.PT1E2}
    \begin{split}
        &\mathbb{P}(\lambda(v_i) = \lambda^i \mbox{ for } i = 0, \dots, M+N) = \sum_{m \geq 0} \sum_{\rho} {\bf 1}\{\rho \preceq \lambda^{i_0-1}, \rho \preceq \lambda^{i_0+1} \} \\
        & \times \mathbb{P}(\lambda(v_i) = \lambda^i \mbox{ for } i = 0, \dots, M+N, \lambda(a_0-1,b_0-1) = \rho, w_{a_0,b_0} = m).
    \end{split}
\end{equation}
If F$^1(\rho, m)$ is the output of the algorithm from Section \ref{Section3.1} for $\nu = \lambda^{i_0-1}$ and $\mu = \lambda^{i_0+1}$, we have by the induction hypothesis (applied to $\gamma'$) and independence
\begin{equation}\label{Eq.PT1E3}
\begin{split}
&\mathbb{P}(\lambda(v_i) = \lambda^i \mbox{ for } i = 0, \dots, M+N, \lambda(a_0-1,b_0-1) = \rho, w_{a_0,b_0} = m)  \\
& = {\bf 1}\{ \mbox{F$^1(\rho, m) = \lambda^{i_0}$}\} \cdot  {\bf 1}\left\{\lambda^0 = \lambda^{M+N} = \emptyset \right\} \cdot (1-x_{a_0}y_{b_0})(x_{a_0} y_{b_0})^m  \\
& \times  \prod_{(a,b) \in Y(\gamma')} \hspace{-1mm} (1-x_ay_b) \cdot \prod_{\substack{i = 1\\ i \not \in \{i_0, i_0+1\}}}^{M+N} \hspace{-1mm} Q(\lambda^{i-1}, \lambda^{i}) \cdot s_{\lambda^{i_0-1}/\rho}(y_{b_0}) s_{\lambda^{i_0+1}/\rho}(x_{a_0}). 
\end{split}
\end{equation}
We mention that in deriving the last identity we used (\ref{Eq.PT1E1}) and that $w_{a_0,b_0}  \sim \mathrm{Geom}(x_{a_0}y_{b_0})$.

From (\ref{Eq.PT1E2}) and (\ref{Eq.PT1E3}), we see that to prove (\ref{Eq.FSDistr}), it suffices to show that 
\begin{equation}\label{Eq.PT1E4}
\begin{split}
&\sum_{m \geq 0} \sum_{\rho} {\bf 1}\{ \mbox{F$^1(\rho, m) = \lambda^{i_0}$}\}  {\bf 1}\{\rho \preceq \lambda^{i_0-1}, \rho \preceq \lambda^{i_0+1} \} (x_{a_0} y_{b_0})^m  \\
& \times s_{\lambda^{i_0-1}/\rho}(y_{b_0}) s_{\lambda^{i_0+1}/\rho}(x_{a_0}) =  s_{\lambda^{i_0}/\lambda^{i_0-1}}(x_{a_0}) s_{\lambda^{i_0}/\lambda^{i_0+1}}(y_{b_0}). 
\end{split}
\end{equation}
Note that if $\rho \preceq \lambda^{i_0-1}, \rho \preceq \lambda^{i_0+1}$, we have that F$^1(\rho,m) \succeq \lambda^{i_0-1}$ and F$^1(\rho,m) \succeq \lambda^{i_0+1}$ for any $m \geq 0$. The latter and (\ref{Eq.DefSchur}) show that both sides of (\ref{Eq.PT1E4}) are zero unless $\lambda^{i_0} \succeq \lambda^{i_0 - 1}$ and $\lambda^{i_0} \succeq \lambda^{i_0+1}$. On the other hand, if $\lambda^{i_0} \succeq \lambda^{i_0 - 1}$ and $\lambda^{i_0} \succeq \lambda^{i_0+1}$, then, since F$^1$ is a bijection, exactly one summand on the top line of (\ref{Eq.PT1E4}) is non-zero, corresponding to a unique pair $(m,\rho)$. Combining the latter with (\ref{Eq.DefSchur}) and (\ref{Eq.MassPres}), we see that both lines of (\ref{Eq.PT1E4}) are equal to $x_{a_0}^Ay_{b_0}^B$ with 
$$A = \sum_{i \geq 1} (\lambda^{i_0}_i - \lambda^{i_0-1}_i), \mbox{ and }B = \sum_{i \geq 1} (\lambda^{i_0}_i - \lambda^{i_0+1}_i).$$
This completes the proof of (\ref{Eq.PT1E4}), and hence the induction step.

%
%
\subsection{Proof of Theorem \ref{T2}}\label{Section4.2} We proceed to prove (\ref{Eq.HSDistr}) by induction on $|Y(\gamma)|$. For the base case $|Y(\gamma)| = 0$, we have $N = 0$ and the word corresponding to $\gamma$ is 
\begin{equation*}
s_1 \dots s_{M+N} = \underbrace{R, \dots, R}_\text{$M$ times}.
\end{equation*}
By Proposition \ref{P.GreeneHS} we have $\mathbb{P}(\lambda(v_i) = \emptyset \mbox{ for } i = 0, \dots, M+N) = 1$, which implies (\ref{Eq.HSDistr}). 

We now fix $k \geq 1$, suppose that (\ref{Eq.HSDistr}) holds when $|Y(\gamma)| \leq k-1$, and proceed to prove it when $|Y(\gamma)| = k$. Since $|Y(\gamma)| = k \geq 1$, we know that $N \geq 1$, and the word corresponding to $\gamma$ contains $N$ letters $D$, and $M$ letters $R$.

Suppose the word corresponding to $\gamma$ is not as in (\ref{Eq.Extreme}), and so must contain $s_{i_0}s_{i_0+1} = RD$, for some $i_0 \in \{1, \dots, M+N-1\}$. Let $\gamma'$ be the down-right path starting from $(N,N)$, whose word is the same as $\gamma$, but with $s_{i_0}s_{i_0+1} = DR$ instead. Note that $Y(\gamma) = Y(\gamma') \cup \{(a_0,b_0)\}$ for some $(a_0,b_0)$ with $1 \leq b_0 \leq N < a_0$. In addition, setting $\gamma' = (v_0', \dots, v_{M+N}')$, we have that (\ref{Eq.PT1E1}) and (\ref{Eq.PT1E2}) hold. By the induction hypothesis (applied to $\gamma'$) and independence we have the following analogue of (\ref{Eq.PT1E3})
\begin{equation*}
\begin{split}
&\mathbb{P}(\lambda(v_i) = \lambda^i \mbox{ for } i = 0, \dots, M+N, \lambda(a_0-1,b_0-1) = \rho, w_{a_0,b_0} = m)  \\
& = {\bf 1}\{ \mbox{F$^1(\rho, m) = \lambda^{i_0}$}\} \cdot \tau_{\lambda^0}(c) \cdot {\bf 1}\left\{\lambda^{M+N} = \emptyset \right\} \cdot \prod_{i = 1}^N(1-cx_i)\prod_{(a,b) \in Y_{<}(\gamma')} \hspace{-3mm} (1-x_ax_b)  \\
& \times(1-x_{a_0}x_{b_0})(x_{a_0} x_{b_0})^m \cdot  \prod_{\substack{i = 1\\ i \not \in \{i_0, i_0+1\}}}^{M+N} Q(\lambda^{i-1}, \lambda^{i}) \cdot s_{\lambda^{i_0-1}/\rho}(x_{b_0}) s_{\lambda^{i_0+1}/\rho}(x_{a_0}). 
\end{split}
\end{equation*}
Combining the last identity with (\ref{Eq.PT1E2}) and (\ref{Eq.PT1E4}) (with $y_{b_0} = x_{b_0}$), we get (\ref{Eq.HSDistr}).\\

In the remainder we suppose $\gamma$ is as in (\ref{Eq.Extreme}). Let $\gamma' = (v_0', \dots,v_{M+N}')$ be such that $v_0' = (N-1,N-1)$ and $v_i' = v_i$ for $i \neq 0$. Note that $Y(\gamma) = Y(\gamma') \cup \{(N,N)\}$. By countable additivity and throwing away terms of zero probability:
\begin{equation}\label{Eq.PT2E2}
    \begin{split}
        &\mathbb{P}(\lambda(v_i) = \lambda^i \mbox{ for } i = 0, \dots, M+N) = \sum_{m \geq 0} \sum_{\rho} {\bf 1}\{\rho \preceq \lambda^{1} \} \\
        & \times \mathbb{P}(\lambda(v_i) = \lambda^i \mbox{ for } i = 0, \dots, M+N, \lambda(N-1,N-1) = \rho, w_{N,N} = m).
    \end{split}
\end{equation}
If F$^1(\rho, m)$ is the output of the algorithm from Section \ref{Section3.1} for $\nu = \mu = \lambda^1$, we have by the induction hypothesis (applied to $\gamma'$) and independence
\begin{equation}\label{Eq.PT2E3}
\begin{split}
&\mathbb{P}(\lambda(v_i) = \lambda^i \mbox{ for } i = 0, \dots, M+N, \lambda(N-1,N-1) = \rho, w_{N,N} = m)  \\
& = {\bf 1}\{ \mbox{F$^1(\rho, m) = \lambda^{0}$}\} \cdot  {\bf 1}\left\{\lambda^{M+N} = \emptyset \right\} \cdot \prod_{i = 1}^{N-1}(1-cx_i) \prod_{(a,b) \in Y_{<}(\gamma')} \hspace{-2mm} (1-x_ax_b)  \\
& \times(1-cx_{N})(c x_{N})^m \cdot \tau_{\rho}(c) \cdot s_{\lambda^{1}/\rho}(x_{N}) \cdot  \prod_{i = 2}^{M+N} Q(\lambda^{i-1}, \lambda^{i}). 
\end{split}
\end{equation}
We mention that in deriving the last identity we used $w_{N,N}  \sim \mathrm{Geom}(cx_{N})$.

From (\ref{Eq.PT2E2}) and (\ref{Eq.PT2E3}), we see that to prove (\ref{Eq.HSDistr}), it suffices to show that 
\begin{equation}\label{Eq.PT2E4}
\begin{split}
&\sum_{m \geq 0} \sum_{\rho} {\bf 1}\{ \mbox{F$^1(\rho, m) = \lambda^{0}$}\}  {\bf 1}\{\rho \preceq \lambda^{1} \} (c x_N)^m  \tau_{\rho}(c) s_{\lambda^{1}/\rho}(x_N) \\
& =  \tau_{\lambda^0}(c) s_{\lambda^{0}/\lambda^{1}}(x_N). 
\end{split}
\end{equation}
Note that if $\rho \preceq \lambda^{1}$, then F$^1(\rho,m) \succeq \lambda^{1}$ for any $m \geq 0$. The latter and (\ref{Eq.DefSchur}) show that both sides of (\ref{Eq.PT2E4}) are zero unless $\lambda^{0} \succeq \lambda^{1}$. On the other hand, if $\lambda^{0} \succeq \lambda^{1}$, then, since F$^1$ is a bijection, we see that only one summand on the left side of (\ref{Eq.PT2E4}) is non-zero, corresponding to a unique pair $(m, \rho)$. Combining the latter with (\ref{Eq.DefSchur}), (\ref{Eq.MassPres}), (\ref{Eq.MassPres2}) and $\tau_{\lambda}(c) = c^{\lambda_1 - \lambda_2 + \lambda_3 - \lambda_4 + \cdots}$, we see that both sides of (\ref{Eq.PT2E4}) are equal to $c^Ax_N^B$ with 
$$A = \sum_{i \geq 1} (\lambda^{0}_{2i-1} - \lambda^{0}_{2i}), \mbox{ and }B = \sum_{i \geq 1} (\lambda^{0}_i - \lambda^{1}_i).$$
This completes the proof of (\ref{Eq.PT2E4}), and hence the induction step.

\begin{appendix}
%
%
\section{Up-right paths and NE-chains}\label{SectionA} The goal of this section is to prove Lemma \ref{L.GequalsH}. In Section \ref{SectionA.1} we introduce some basic definitions and notation about partitions, and in Section \ref{SectionA.2} we prove several auxiliary lemmas that will be needed. The proof of Lemma \ref{L.GequalsH} is given in Section \ref{SectionA.3}.

%
%
\subsection{Basic notation}\label{SectionA.1} A {\em partition} is a sequence $\lambda = (\lambda_1, \lambda_2, \dots)$ of non-negative integers such that $\lambda_1 \geq \lambda_2 \geq \cdots$ and all but finitely many terms are zero. We denote the set of all partitions by $\mathbb{Y}$. The {\em length} of a partition, denoted $\ell(\lambda)$, is the number of non-zero $\lambda_i$, and the {\em weight} of a partition is given by $|\lambda| = \lambda_1 + \lambda_2 + \cdots$. A {\em Young diagram} is a graphical representation of a partition $\lambda$, with $\lambda_1$ left-justified boxes on the top row, $\lambda_2$ boxes on the second row and so on. In general, we do not distinguish between a partition $\lambda$ and the Young diagram representing it. We will graphically draw the boxes belonging to a Young diagram in the fourth quadrant $\mathbb{Z}_{>0} \times \mathbb{Z}_{<0}$, so that the boxes corresponding to the $k$-th row of $\lambda$ have coordinates $(1, -k), (2, -k), \dots, (\lambda_k,-k)$, see Figure \ref{Fig.YoungDiagram}. This convention will be convenient for us once we start discussing up-right paths and NE-chains.

For two diagrams $\lambda$ and $\mu$, we can view them as subsets of $\mathbb{Z}_{>0} \times \mathbb{Z}_{<0}$, and write $\mu \subseteq \lambda$ for the usual inclusion of sets -- it is equivalent to the statement $\lambda_k \geq \mu_k$ for all $k \geq 1$. Given a partition $\lambda$, we denote by $\mathring{\lambda}$ its {\em interior}, which is the set of boxes $(i,j) \in \lambda$, such that $(i+1, j-1) \in \lambda$. We denote by $\partial \lambda$ its {\em boundary}, which is the set of boxes $(i,j) \in \lambda$, such that $(i+1,j-1) \not \in \lambda$. Note that $\mathring{\lambda}$ is itself a Young diagram and $\mathring{\lambda} = \lambda \setminus \partial \lambda$. Given two partitions $\lambda, \mu$, we let $\min(\lambda, \mu)$ denote the Young diagram $\nu$, with $\nu_i = \min(\lambda_i, \mu_i)$. Similarly, we let $\max(\lambda, \mu)$ denote the Young diagram $\rho$, with $\rho_i = \max(\lambda_i, \mu_i)$. Notice that $\min(\lambda, \mu) = \lambda \cap \mu$ is the largest Young diagram that is contained in both $\lambda$ and $\mu$, in the sense that if $\tilde{\nu} \subseteq \lambda$ and $\tilde{\nu} \subseteq \mu$, then $\tilde{\nu} \subseteq \min(\lambda, \mu)$. Similarly, $\max(\lambda, \mu) = \lambda \cup \mu$ is the smallest diagram that contains both $\lambda$ and $\mu$. Given a finite subset $S \subset \mathbb{Z}_{>0} \times \mathbb{Z}_{<0}$, we denote by $\lambda(S) = \cap_{\lambda \in \mathbb{Y}: S \subseteq \lambda} \lambda$, which is the smallest Young diagram containing $S$. Some of these definitions are illustrated in Figure \ref{Fig.YoungDiagram}. 

\begin{figure}[tb]
\centering
\begin{tikzpicture}[scale=0.45, every node/.style={scale=0.9}]

\definecolor{intColor}{gray}{0.65}


\foreach \i/\row in {1/5, 2/3, 3/2, 4/1} {
  \foreach \j in {1,...,\row} {
    \draw (\j+13,-\i) rectangle ++(1,-1);
  }
}

\foreach \i/\row in {1/3, 2/3, 3/1} {
  \foreach \j in {1,...,\row} {
    \draw (\j+21,-\i) rectangle ++(1,-1);
  }
}

\node at (3, -6) {\(\lambda = (5,3,1)\)};
\node at (9.5, -6) {\(\mu = (3,3,2,1)\)};
\node at (16.5, -6) {\(\max(\lambda, \mu) = (5,3,2,1)\)};
\node at (24, -6) {\(\min(\lambda, \mu) = (3,3,1)\)};

\foreach \x/\y in {1/1, 2/1} {
  \fill[intColor] (\x,-\y) rectangle ++(1,-1);
  \draw (\x,-\y) rectangle ++(1,-1);
}

\foreach \x/\y in {1/1, 2/1, 1/2} {
  \fill[intColor] (7 + \x,-\y) rectangle ++(1,-1);
  \draw (7 + \x,-\y) rectangle ++(1,-1);
}

\foreach \x/\y in {5/1, 4/1, 3/1, 3/2, 2/2, 1/2, 1/3} {
  \draw[very thick, dashed] (\x,-\y) rectangle ++(1,-1);
}

\foreach \x/\y in {3/1, 3/2, 2/2, 2/3, 1/3, 1/4} {
  \draw[very thick, dashed] (7+\x,-\y) rectangle ++(1,-1);
}

\foreach \j in {1,...,5} {
  \node at (\j+0.5, -0.5) {\(\j\)};
}
\foreach \i in {1,...,3} {
  \node at (0.2, -\i - 0.5) {\( -\i\)};
}

\foreach \j in {1,...,3} {
  \node at (\j+0.5 + 7, -0.5) {\(\j\)};
}
\foreach \i in {1,...,4} {
  \node at (0.2 + 7, -\i - 0.5) {\( -\i\)};
}

\foreach \j in {1,...,5} {
  \node at (\j+0.5 + 13, -0.5) {\(\j\)};
}
\foreach \i in {1,...,4} {
  \node at (0.2 + 13, -\i - 0.5) {\( -\i\)};
}

\foreach \j in {1,...,3} {
  \node at (\j+0.5 + 21, -0.5) {\(\j\)};
}
\foreach \i in {1,...,3} {
  \node at (0.2 + 21, -\i - 0.5) {\( -\i\)};
}

\end{tikzpicture}
\caption{Young diagrams for $\lambda = (5,3,1)$, $\mu = (3,3,2,1)$, and their maximum $\max(\lambda, \mu)$ and minimum $\min(\lambda, \mu)$. The interiors $\mathring{\lambda}$ and $\mathring{\mu}$ are in gray; the boundaries $\partial \lambda$ and $\partial \mu$ are outlined with dashed boxes.}\label{Fig.YoungDiagram}
\end{figure}

Finally, recall from Section \ref{Section2} that a {\em north-east (NE)-chain} $\chi$ is a (possibly empty) sequence $\chi = ((a_1,b_1), \dots, (a_r,b_r))$ with $a_i, b_i \in \mathbb{Z}$, $a_{i-1} \leq a_i$, $b_{i-1} \leq b_i$ and $(a_{i-1}, b_{i-1}) \neq (a_i, b_i)$. We also recall that an {\em up-right path} $\pi$ in $\mathbb{Z}^2$ is a (possibly empty) sequence of vertices $\pi = (v_1, \dots, v_r)$ with $v_i \in \mathbb{Z}^2$ and $v_i - v_{i-1} \in \{(0,1), (1,0)\}$. In Section \ref{SectionA.2} we exclusively work with NE-chains and up-right paths whose vertices lie in $\mathbb{Z}_{>0} \times \mathbb{Z}_{<0}$, and do not mention this further.

%
%
\subsection{Auxiliary lemmas}\label{SectionA.2} We continue with the same notation as in Section \ref{SectionA.1}. Our goal is to prove two lemmas, see Lemmas \ref{Lem.Layers} and \ref{Lem.OffDiag}, that will be useful in the proof of Lemma \ref{L.GequalsH}. In order to establish those, we start by proving a few simple lemmas.

\begin{lemma}\label{Lem.ChainInBoundary} Let $\chi$ be a NE-chain and $\lambda(\chi)$ the smallest diagram containing it. Then $\chi \subseteq \partial \lambda(\chi)$. 
\end{lemma}
\begin{proof} Fix $(a,b) \in \chi$ and note $(a,b) \in \lambda(\chi)$ by definition. For each $(i,j) \in \chi$, we have $i \leq a$ and $j \leq b$, or $i \geq a$ and $j \geq b$. Consequently, $\chi \subset \mathbb{Z}_{>0} \times \mathbb{Z}_{<0} \setminus T$, where $T = \{(i,j): i > a \mbox{ and } j < b\}$. If $\mu = \lambda(\chi) \cap (\mathbb{Z}_{>0} \times \mathbb{Z}_{<0} \setminus T)$, then we see that $\mu$ is a diagram that contains $\chi$, and so $\mu = \lambda(\chi)$ by minimality. This shows $(a+1,b-1) \not \in \lambda(\chi)$, and so $(a,b) \in \partial \lambda(\chi)$ as desired.
\end{proof}

\begin{lemma}\label{Lem.UnionBoundaries} Let $\lambda, \mu \in \mathbb{Y}$, and set $\nu = \min(\lambda, \mu)$, $\rho = \max(\lambda, \mu)$. Then $\partial \lambda \cup \partial \mu = \partial \nu \cup \partial \rho$.  
\end{lemma}
\begin{proof} Let $(a,b) \in \partial \lambda$. If $(a,b) \in \mu$, then $(a,b) \in \nu$. As $(a+1,b-1) \not \in \lambda$, we have $(a+1,b-1) \not \in \nu$ and so $(a,b) \in \partial \nu \subseteq \partial \nu \cup \partial \rho$. Conversely, if $(a,b) \not \in \mu$, then $(a+1,b-1) \not \in \mu$. As $(a,b) \in \partial \lambda$, we also have $(a+1,b-1) \not \in \lambda$, and so $(a+1, b-1) \not \in \rho$. As $(a,b) \in \lambda \subseteq \rho$, we conclude $(a,b) \in \partial \rho \subseteq \partial \nu \cup \partial \rho$. As $(a,b) \in \partial \lambda$ was arbitrary, we conclude $\partial \lambda \subseteq \partial \nu \cup \partial \rho$. Swapping the roles of $\lambda$ and $\mu$ shows $\partial \mu \subseteq \partial \nu \cup \partial \rho$, and so 
\begin{equation}\label{Eq.Cont1}
\partial \lambda \cup \partial \mu \subseteq \partial \nu \cup \partial \rho.  
\end{equation}

Suppose now that $(a,b) \in \partial \rho$. Then $(a,b) \in \lambda$, or $(a,b) \in \mu$, or both. Without loss of generality, suppose that $(a,b) \in \lambda$, and note that $(a+1,b-1) \not \in \rho$ implies $(a+1,b-1) \not \in \lambda$, so that $(a,b) \in \partial \lambda$. This shows $(a,b) \in \partial \lambda \subseteq \partial \lambda \cup \partial \mu$. As $(a,b) \in \partial \rho$ was arbitrary, we conclude $\partial \rho \subseteq \partial \lambda \cup \partial \mu$. 

Finally, suppose that $(a,b) \in \partial \nu$, and hence $(a,b) \in \lambda \cap \mu$. If $(a,b) \in \mathring{\lambda} \cap \mathring{\mu}$, then $(a+1,b-1) \in \lambda \cap \mu = \nu$, contradicting $(a,b) \in \partial \nu$. Consequently, $(a,b) \not \in \mathring{\lambda}$ (in which case $(a,b) \in \partial \lambda$), or $(a,b) \not \in \mathring{\mu}$ (in which case $(a,b) \in \partial \mu$), or both. Either way, we conclude $(a,b) \in \partial \lambda \cup \partial \mu$. As $(a,b) \in \partial \nu$ was arbitrary, we conclude $\partial \nu \subseteq \partial \lambda \cup \partial \mu$. The last two paragraphs now give
\begin{equation}\label{Eq.Cont2}
\partial \nu \cup \partial \rho \subseteq \partial \lambda \cup \partial \mu.  
\end{equation}
Equations (\ref{Eq.Cont1}) and (\ref{Eq.Cont2}) give the statement of the lemma.
\end{proof}

\begin{lemma}\label{Lem.Twist} Let $\chi_1, \chi_2$ be disjoint NE-chains and set $\lambda = \max(\lambda(\chi_1), \lambda(\chi_2))$. We can find two disjoint NE-chains $\chi'_1, \chi'_2$, such that: 
$$\mbox{ {\em (1) $\chi_1', \chi_2' \subseteq \lambda$; (2) $\chi_1' \sqcup \chi_2' = \chi_1 \sqcup \chi_2$; (3) $\chi_1' \subseteq \partial \lambda$; (4) $\chi_2' \subseteq \mathring{\lambda}$; (5) $\lambda(\chi'_1) = \lambda$. }}$$
\end{lemma}
\begin{proof} Let $\mu = \min(\lambda(\chi_1), \lambda(\chi_2))$, and define
$$\chi_1' = \partial \lambda \cap (\chi_1 \sqcup \chi_2) \mbox{ and } \chi_2' = (\partial \mu \setminus \partial \lambda) \cap (\chi_1 \sqcup \chi_2).$$
The latter specifies our choice of $\chi_1'$ and $\chi_2'$, and we proceed to verify that they satisfy the conditions in the lemma. Note that $\chi_1'$ and $\chi_2'$ are NE-chains, since they are subsets of the NE-chains $\partial \lambda$ and $\partial \mu$, respectively. As $\chi_1' \subseteq \partial \lambda$, $\chi_2' \subseteq \partial \mu \setminus \partial \lambda$, and $\mu \subseteq \lambda$, we trivially conclude that $\chi_1', \chi_2'$ are disjoint and satisfy conditions (1), (3) and (4). 

We next check condition (2), and by construction we only need to show $(a,b) \in \chi_1' \sqcup \chi_2'$ for all $(a,b) \in \chi_1 \sqcup \chi_2$. Suppose $(a,b) \in \chi_1 \sqcup \chi_2$. From Lemmas \ref{Lem.ChainInBoundary} and \ref{Lem.UnionBoundaries}, we have $(a,b) \in \partial \lambda \cup \partial \mu$. If $(a,b) \in \partial \lambda$, then $(a,b) \in \chi_1'$, while if $(a,b) \not \in \partial \lambda$, then $(a,b) \in \partial \mu \setminus \partial \lambda$, and so $(a,b) \in \chi_2'$. Either way, we conclude $(a,b) \in \chi'_1 \sqcup \chi'_2$.

Lastly, we verify condition (5). Since $\chi_1' \subseteq \lambda$, we have $\lambda(\chi_1') \subseteq \lambda$. To show the converse inclusion, we seek to show that for all $i \geq 1$, we have $\lambda_i \leq \lambda_i(\chi_1')$. The latter is clear if $\lambda_i = 0$, and so we may assume that $\lambda_i > 0$. In particular, we see $\mathring{\lambda}_i \leq \lambda_i - 1$, and so by condition (4) we have $\lambda_i(\chi_2') \leq \lambda_i - 1$. Set $\nu = \max(\lambda(\chi_1'), \lambda(\chi_2'))$. From condition (2), we have $\lambda \subseteq \nu$, and so $\nu_i \geq \lambda_i$. But $\nu_i = \max(\lambda_i(\chi_1'),\lambda_i(\chi_2')) \leq \max(\lambda_i(\chi_1'), \lambda_i - 1)$ and so $\lambda_i \leq \lambda_i(\chi_1')$ as desired.
\end{proof}

The above three lemmas will be used to establish the following two results, which are the ones we require in Section \ref{SectionA.3} below. The first of these results says that any $k$ pairwise disjoint NE-chains can be covered by $k$ pairwise disjoint nested diagram boundaries, see Figure \ref{Fig.Induction}. Its proof relies on repeated uses of Lemma \ref{Lem.Twist}.

\begin{figure}[ht]
\centering
\begin{tikzpicture}[scale=0.55]

\definecolor{Bg}{gray}{1.0}        
\definecolor{C1}{gray}{0.8} 
\definecolor{C2}{gray}{0.5} 
\definecolor{C3}{gray}{0.2} 

\begin{scope}[shift={(0,0)}]


  \node at (3, -0.5 ) {\(m = 0\)};

  \fill[C1] (4,-3) rectangle ++(1,-1);
  \fill[C1] (1,-3) rectangle ++(1,-1);
  \fill[C1] (5,-1) rectangle ++(1,-1);

  \fill[C2] (1,-2) rectangle ++(1,-1);
  \fill[C2] (4,-1) rectangle ++(1,-1);
  \fill[C2] (3,-2) rectangle ++(1,-1);
  \fill[C2] (4,-2) rectangle ++(1,-1);
  
  \fill[C3] (2,-4) rectangle ++(1,-1);
  \fill[C3] (2,-1) rectangle ++(1,-1);
  \fill[C3] (2,-2) rectangle ++(1,-1);

  \foreach \row [count=\i] in {5,4,4,3,2} {
    \foreach \j in {1,...,\row} {
      \draw (\j,-\i) rectangle ++(1,-1);
    }
  }

  \draw[dashed, line width=2pt]
    (1,-5) -- (3, -5) -- (3,-1) -- (1,-1) -- (1,-5);
\end{scope}

\begin{scope}[shift={(7.5,0)}]

  \node at (3, -0.5) {\(m = 1\)};

  \fill[C1] (1,-3) rectangle ++(1,-1);
  \fill[C1] (2,-1) rectangle ++(1,-1);
  \fill[C1] (2,-2) rectangle ++(1,-1);

   \fill[C2] (1,-2) rectangle ++(1,-1);
  \fill[C2] (4,-1) rectangle ++(1,-1);
  \fill[C2] (3,-2) rectangle ++(1,-1);
  \fill[C2] (4,-2) rectangle ++(1,-1);
  
  \fill[C3] (2,-4) rectangle ++(1,-1);
  \fill[C3] (4,-3) rectangle ++(1,-1);
  \fill[C3] (5,-1) rectangle ++(1,-1);

  \foreach \row [count=\i] in {5,4,4,3,2} {
    \foreach \j in {1,...,\row} {
      \draw (\j,-\i) rectangle ++(1,-1);
    }
  }

  \draw[dashed, line width=2pt]
    (1,-5) -- (3, -5) -- (3, -4) -- (5,-4) -- (5,-2) -- (6,-2) -- (6,-1) -- (1,-1) -- (1,-5);

\end{scope}

\begin{scope}[shift={(15,0)}]

  \node at (3, -0.5) {\(m = 2\)};

  \fill[C1] (1,-3) rectangle ++(1,-1);
  \fill[C1] (2,-1) rectangle ++(1,-1);
  \fill[C1] (2,-2) rectangle ++(1,-1);

  \fill[C2] (1,-2) rectangle ++(1,-1);
  \fill[C2] (3,-2) rectangle ++(1,-1);

  \fill[C3] (2,-4) rectangle ++(1,-1);
  \fill[C3] (4,-3) rectangle ++(1,-1);
  \fill[C3] (5,-1) rectangle ++(1,-1);
  \fill[C3] (4,-2) rectangle ++(1,-1);
    \fill[C3] (4,-1) rectangle ++(1,-1);

  \foreach \row [count=\i] in {5,4,4,3,2} {
    \foreach \j in {1,...,\row} {
      \draw (\j,-\i) rectangle ++(1,-1);
    }
  }

  \draw[dashed, line width=2pt]
    (1,-5) -- (3, -5) -- (3, -4) -- (5,-4) -- (5,-2) -- (6,-2) -- (6,-1) -- (1,-1) -- (1,-5);

\end{scope}

\begin{scope}[shift={(15,-6)}]

  \fill[C1] (2,-1) rectangle ++(1,-1);

  \fill[C2] (1,-2) rectangle ++(1,-1);
  \fill[C2] (3,-2) rectangle ++(1,-1);
    \fill[C2] (1,-3) rectangle ++(1,-1);
  \fill[C2] (2,-2) rectangle ++(1,-1);

  \fill[C3] (2,-4) rectangle ++(1,-1);
  \fill[C3] (4,-3) rectangle ++(1,-1);
  \fill[C3] (5,-1) rectangle ++(1,-1);
  \fill[C3] (4,-2) rectangle ++(1,-1);
    \fill[C3] (4,-1) rectangle ++(1,-1);

  \foreach \row [count=\i] in {5,4,4,3,2} {
    \foreach \j in {1,...,\row} {
      \draw (\j,-\i) rectangle ++(1,-1);
    }
  }

  \draw[ line width=2pt]
    (1,-5) -- (3, -5) -- (3, -4) -- (5,-4) -- (5,-2) -- (6,-2) -- (6,-1) -- (1,-1) -- (1,-5);
  \draw[ line width=2pt]
    (1,-4) -- (2, -4) -- (2, -3) -- (4,-3) -- (4,-1) ;
  \draw[ line width=2pt]
    (1,-2) -- (3, -2) -- (3, -1);

\end{scope}

\begin{scope}[shift={(0,-8)}]


  \fill[C1] (1,-3) rectangle ++(1,-1);
  \fill[C1] (2,-1) rectangle ++(1,-1);
  \fill[C1] (2,-2) rectangle ++(1,-1);

  \fill[C2] (1,-2) rectangle ++(1,-1);
  \fill[C2] (3,-2) rectangle ++(1,-1);

  \foreach \row [count=\i] in {3,3,1} {
    \foreach \j in {1,...,\row} {
      \draw (\j,-\i) rectangle ++(1,-1);
    }
  }

  \draw[->, very thick] (5,-2) -- (6,-2);

    \fill[C1] (8,-1) rectangle ++(1,-1);

  \fill[C2] (7,-2) rectangle ++(1,-1);
  \fill[C2] (9,-2) rectangle ++(1,-1);
  \fill[C2] (7,-3) rectangle ++(1,-1);
  \fill[C2] (8,-2) rectangle ++(1,-1);

  \foreach \row [count=\i] in {3,3,1} {
    \foreach \j in {1,...,\row} {
      \draw (\j + 6,-\i) rectangle ++(1,-1);
    }
  }
  \draw[ line width=2pt]
    (7,-4) -- (8, -4) -- (8, -3) -- (10,-3) -- (10,-1) -- (7,-1) -- (7,-4) ;
  \draw[ line width=2pt]
    (7,-2) -- (9, -2) -- (9, -1);

\end{scope}

\begin{scope}[shift={(1,-7)}]
  \draw[fill=C1] (0,0) rectangle ++(0.5,-0.5);
  \node[anchor=west] at (0.7,-0.25) {$\chi^m_1$};
  \draw[fill=C2] (2.2,0) rectangle ++(0.5,-0.5);
  \node[anchor=west] at (2.9,-0.25) {$\chi^m_2$};
  \draw[fill=C3] (4.4,0) rectangle ++(0.5,-0.5);
  \node[anchor=west] at (5.1,-0.25) {$\chi^m_3$};
\end{scope}

\end{tikzpicture}
\caption{The top-left depicts $\lambda = (5,4,4,3,2)$ with three pairwise disjoint NE-chains $\chi_1,\chi_2, \chi_3$ drawn inside. The top three diagrams show the triplet of NE-chains $(\chi^m_1, \chi^m_2, \chi^m_3)$ in the proof of Lemma \ref{Lem.Layers} for $m = 0,1,2$. The dashed lines indicate $\lambda(\chi^m_3)$ for $ m = 0,1,2$. The bottom-left depicts the result of applying the induction hypothesis to $(\chi^3_1, \chi^3_2)$. The bottom-right depicts the $\lambda^1, \lambda^2, \lambda^3$ from Lemma \ref{Lem.Layers} in bold lines, and we have $\lambda^1 = (2)$, $\lambda^2 = (3,3,1)$, $\lambda^3 = (5,4,4,2)$.}\label{Fig.Induction}
\end{figure}

\begin{lemma}\label{Lem.Layers} Fix a Young diagram $\lambda$, and $k \in \mathbb{N}$. Let $\chi_1, \dots, \chi_k$ be $k$ pairwise disjoint NE-chains with $\chi_i \subseteq \lambda$ for $i = 1, \dots, k$. We can find $k$ diagrams $\lambda^1 \subseteq \lambda^{2} \subseteq \cdots \subseteq \lambda^{k} \subseteq  \lambda$, such that 
\begin{equation}\label{Eq.Layers1}
\lambda^{i} \subseteq \mathring{\lambda}^{i+1} \mbox{ for }i = 1, \dots, k-1;
\end{equation}
\begin{equation}\label{Eq.Layers2}
\sqcup_{i = 1}^k \chi_i \subseteq \sqcup_{i = 1}^k \partial \lambda^i.
\end{equation}
\end{lemma}
\begin{proof} We proceed to verify the statement by induction on $k$. When $k = 1$, we can take $\lambda^1 = \lambda(\chi_1)$, and then (\ref{Eq.Layers1}) is vacuously satisfied, while (\ref{Eq.Layers2}) holds by Lemma \ref{Lem.ChainInBoundary}. In the sequel we assume $k \geq 2$ and that the result holds for $k-1$.

We define sequences of chains $(\chi^m_1, \dots, \chi^m_{k})$ for $m = 0, \dots, k$, as follows. Start by setting $\chi^{0}_{i} = \chi_i$ for $i = 1, \dots, k$. If $(\chi^m_1, \dots, \chi^m_{k})$ for some $m \leq k-1$ have been constructed, we let $\chi^{m+1}_i = \chi^m_i$ for $i \not \in \{m+1, k\}$. In addition, if $\chi_1', \chi_2'$ are as in Lemma \ref{Lem.Twist} for $\chi_1 = \chi^m_{k}$ and $\chi_2 = \chi^m_{m+1}$, we let $\chi^{m+1}_{m+1} = \chi_2'$ and $\chi^{m+1}_{k} = \chi_1'$. From Lemma \ref{Lem.Twist}[(1) and (2)], we observe that for $m \geq 0$, we have  
$$\mbox{(I) $\chi_i^m \cap \chi_j^m = \emptyset$ for $i \neq j$, (II) $\sqcup_{i = 1}^{k} \chi^m_i = \sqcup_{i = 1}^{k} \chi^{m+1}_{i}$, (III) $\sqcup_{i = 1}^{k} \chi_i^m \subseteq \lambda$.}$$
In addition, from Lemma \ref{Lem.Twist}[(5)] we conclude $\lambda(\chi_{k}^m) \subseteq \lambda(\chi^{m+1}_{k})$ for $m = 0, \dots, k-1$. The latter and Lemma \ref{Lem.Twist}[(4)] further imply $\chi_{m+1}^{k} = \chi_{m+1}^{m+1} \subseteq \mathring{\nu}^{m+1} \subseteq \mathring{\nu}^{k}$ for $m = 0, \dots, k-2$, where we have set $\nu^{m} = \lambda(\chi^{m}_{k})$ for $m = 1 , \dots, k$. Finally, from Lemma \ref{Lem.Twist}[(3)] we see that $\chi_{k}^{k} \subseteq \partial \nu^{k}$.

From our work in the previous paragraph, we see that if we set $\lambda^{k} = \nu^{k}$, then $\chi^{k}_{1}, \dots, \chi^{k}_{k-1}$ are pairwise disjoint and contained in $\mathring{\lambda}^{k}$, while $\chi_{k}^{k} \subseteq \partial \lambda^{k}$. Applying the induction hypothesis to $(\chi_1^{k}, \dots, \chi^{k}_{k-1})$ and $\lambda = \mathring{\lambda}^{k}$, we obtain a sequence $\lambda^1 \subseteq \lambda^2 \subseteq \cdots \subseteq \lambda^{k-1} \subseteq \mathring{\lambda}^{k}$, satisfying (\ref{Eq.Layers1}) and (\ref{Eq.Layers2}) with $k$ replaced by $k-1$. If we append $\lambda^{k}$ to this list, we see that it satisfies (\ref{Eq.Layers1}). In addition, from property (III) above, we have $\lambda^{k} = \lambda(\chi^{k}_{k}) \subseteq \lambda$. Lastly, from condition (II) and $\chi_{k}^{k} \subseteq  \partial \lambda^{k}$, we conclude 
$$\sqcup_{i = 1}^{k} \chi_i = \sqcup_{i = 1}^{k} \chi_i^{k} \subseteq \sqcup_{i = 1}^{k-1} \partial \lambda^i \sqcup \partial \lambda^{k},$$
verifying (\ref{Eq.Layers2}). Overall, we see that $\lambda^1, \dots, \lambda^{k}$ satisfy the conditions of the lemma. This completes the induction step, and the general result follows by induction.
\end{proof}

We now turn to the final result of the section, which shows that if we take any maximal collection of $k$ pairwise disjoint NE-chains, contained in an $m \times n$ rectangle $R$, then their union contains the off-diagonal vertices near the corners $(1,-n)$ and $(m,-1)$, see the left side of Figure \ref{Fig.Rectangle}.

\begin{figure}[ht]
\centering
\begin{tikzpicture}[scale=0.63]

\def\m{7} 
\def\n{5} 
\def\k{3} 

\definecolor{Bg}{gray}{1.0}        
\definecolor{C1}{gray}{0.8} 
\definecolor{C2}{gray}{0.5} 
\definecolor{C3}{gray}{0.2} 

\begin{scope}[shift={(0,0)}]


\foreach \i in {1,...,\n}{
  \foreach \j in {1,...,\m}{
    \fill[white] (\j,-\i) rectangle ++(1,-1);
    \draw[black] (\j,-\i) rectangle ++(1,-1);
  }
}

  \foreach \j in {1,...,\m}{
      \node at (\j+0.5, -0.5) {\(\j\)};
  }
  \foreach \j in {1,...,\n}{
      \node at (0.35,-0.5 - \j) {\(-\j\)};
  }

\foreach \j in {1,...,\k}{
  \pgfmathsetmacro\y{-\n + \k - \j}
  \fill[gray!60] (\j,\y) rectangle ++(1,-1);
  \draw[black] (\j,\y) rectangle ++(1,-1);
}

\foreach \j in {1,...,\k}{
  \pgfmathsetmacro\x{\m - \j + 1}
  \pgfmathsetmacro\y{\j - \k - 1}
  \fill[gray!60] (\x,\y) rectangle ++(1,-1);
  \draw[black] (\x,\y) rectangle ++(1,-1);
}
\end{scope}

\begin{scope}[shift={(11,0)}]

\foreach \x/\y in {4/1, 3/1, 3/2, 2/2, 1/2, 1/3} {
  \fill[C3] (\x/2,-\y/2- 3.5) rectangle ++(1/2,-1/2);
}

\foreach \x/\y in {4/1, 4/2, 3/2, 3/3, 3/4, 2/4, 2/5, 1/5} {
  \fill[C2] (\x/2,-\y/2- 1.5) rectangle ++(1/2,-1/2);
}

\foreach \x/\y in {4/1, 3/1, 3/2, 2/2, 1/2, 1/2, 1/3, 1/4, 1/5} {
  \fill[C1] (\x/2,-\y/2- 0.5) rectangle ++(1/2,-1/2);
}

\foreach \i in {1,...,10}{
  \foreach \j in {1,...,4}{
    \draw[black] (\j/2,-\i/2 - 0.5) rectangle ++(1/2,-1/2);
  }
}

\end{scope}

\begin{scope}[shift={(11,-2)}]
  \draw[fill=C1] (-1,0) rectangle ++(0.5,-0.5);
  \node[anchor=west] at (-0.5,-0.3) {$s_1$};
  \draw[fill=C2] (-1,-1) rectangle ++(0.5,-0.5);
  \node[anchor=west] at (-0.5,-1.3) {$s_2$};
  \draw[fill=C3] (-1,-2) rectangle ++(0.5,-0.5);
  \node[anchor=west] at (-0.5,-2.3) {$s_3$};
  \draw[->, very thick] (3,-1.5) -- (3.5,-1.5);

\draw[fill=white, line width=0.5mm] (0.75,0.25) circle (4pt);
\draw[fill=white, line width=0.5mm] (1.25,-1.75) circle (4pt);
\draw[fill=white, line width=0.5mm] (1.75,-2.25) circle (4pt);
  
\end{scope}

\begin{scope}[shift={(16,0)}]

\foreach \x/\y in {4/1, 3/1, 3/2, 3/3, 3/4} {
  \fill[C3] (\x/2,-\y/2- 3.5) rectangle ++(1/2,-1/2);
}

\foreach \x/\y in {4/1, 4/2, 3/2, 3/3, 3/4, 2/4, 2/5, 2/6, 2/7, 2/8} {
  \fill[C2] (\x/2,-\y/2- 1.5) rectangle ++(1/2,-1/2);
}

\foreach \x/\y in {4/1, 3/1, 3/2, 2/2, 1/2, 1/2, 1/3, 1/4, 1/5, 1/6, 1/7, 1/8, 1/9, 1/10} {
  \fill[C1] (\x/2,-\y/2- 0.5) rectangle ++(1/2,-1/2);
}

\foreach \i in {1,...,10}{
  \foreach \j in {1,...,4}{
    \draw[black] (\j/2,-\i/2 - 0.5) rectangle ++(1/2,-1/2);
  }
}

\draw[fill=white, line width=0.5mm] (0.75,-1.75) circle (4pt);
\draw[fill=white, line width=0.5mm] (1.25,-3.75) circle (4pt);
\draw[fill=white, line width=0.5mm] (1.75,-4.25) circle (4pt);

\end{scope}

\begin{scope}[shift={(16,-2)}]
  \draw[fill=C1] (-1,0) rectangle ++(0.5,-0.5);
  \node[anchor=west] at (-0.5,-0.3) {$\tilde{s}_1$};
  \draw[fill=C2] (-1,-1) rectangle ++(0.5,-0.5);
  \node[anchor=west] at (-0.5,-1.3) {$\tilde{s}_2$};
  \draw[fill=C3] (-1,-2) rectangle ++(0.5,-0.5);
  \node[anchor=west] at (-0.5,-2.3) {$\tilde{s}_3$};

\end{scope}

\end{tikzpicture}
\caption{The left side depicts the partition $R$ in Lemma \ref{Lem.OffDiag} and the vertices in (\ref{Eq.OffDiag}) that are contained in $\sqcup_{i = 1}^k \chi_i$ for $m = 7$, $n = 5$ and $k = 3$. The right side depicts the up-right paths $s_1, \dots, s_k$ and $\tilde{s}_1, \dots, \tilde{s}_k$ in the proof of Lemma \ref{Lem.OffDiag}. The circles indicate the locations of the points $(i,u_i)$ for $i = 1, \dots, k$.} \label{Fig.Rectangle}
\end{figure}

\begin{lemma}\label{Lem.OffDiag} Fix integers $m,n,k \in \mathbb{N}$ with $k \leq \min(m,n)$, and let $R = (R_1, R_2, \dots)$ be the partition with $R_i = m \cdot {\bf 1}\{i \leq n\}$. Let $S$ be the set of $k$-tuples $(s_1, \dots, s_k)$ of pairwise disjoint NE-chains that are contained in $R$ with the partial order induced by the set inclusion of the union $\sqcup_{i = 1}^k s_i$. If $(\chi_1, \dots, \chi_k) \in S$ is a maximal element, then for all $j = 1, \dots, k$, we have 
\begin{equation}\label{Eq.OffDiag}
(j,-n + k - j) \in \sqcup_{i = 1}^k \chi_i, \mbox{ and } (m - j + 1, j - k - 1) \in \sqcup_{i = 1}^k \chi_i.
\end{equation}
\end{lemma}
\begin{proof} By symmetry, it suffices to prove the first statement in (\ref{Eq.OffDiag}). Indeed, consider the map $f(a,b) = (m - a + 1, - n - b - 1)$, which defines an involution of $R$ (viewed as a subset of $\mathbb{Z}_{>0} \times \mathbb{Z}_{<0}$). By applying it to each vertex in each $s_i$ for $(s_1, \dots, s_k) \in S$, we obtain an involution of $S$ that respects the partial order. Applying the first statement in (\ref{Eq.OffDiag}) to the image of $(\chi_1, \dots, \chi_k)$ under this involution, gives the second statement in (\ref{Eq.OffDiag}). 

In the remainder we fix a maximal $(\chi_1, \dots, \chi_k)$ and prove the first statement in (\ref{Eq.OffDiag}). Let $\lambda^1, \dots, \lambda^k$ be as in Lemma \ref{Lem.Layers} for our present choice of $\chi_1, \dots, \chi_k$, and $\lambda = R$. Set $s_i = \partial \lambda^i$ for $i = 1, \dots, k$ and note that $(s_1, \dots, s_k) \in S$. In addition, from (\ref{Eq.Layers2}) and the maximality of $(\chi_1, \dots, \chi_k)$, we conclude
\begin{equation}\label{Eq.OD1}
\sqcup_{i = 1}^k \chi_i = \sqcup_{i = 1}^k s_i.
\end{equation}
Consequently, it suffices to show that for all $j = 1, \dots, k$
\begin{equation}\label{Eq.OD2}
(j,-n + k - j) \in \sqcup_{i = 1}^k s_i.
\end{equation}

Suppose for the sake of contradiction that $(j_0,-n + k - j_0) \not \in \sqcup_{i = 1}^k s_i$ for some $j_0 \in \{1, \dots, k\}$. If $s_1 = \emptyset$, we can set $\hat{s}_1 = (j_0,-n + k - j_0)$, and then $(\hat{s}_1, s_2, \dots, s_k) \in S$ with 
$$\sqcup_{i = 1}^k s_i \subsetneq \hat{s}_1 \sqcup \sqcup_{i = 2}^k s_i.$$
The latter and (\ref{Eq.OD1}) contradict the maximality of $(\chi_1, \dots, \chi_k)$, and so $s_1 \neq \emptyset$. From (\ref{Eq.Layers1}) we conclude that $s_i \neq \emptyset$ for all $i = 1, \dots, k$. 

Let us enumerate the vertices in the up-right paths $s_i$ as usual by $s_i = ((a^i_1,b^i_1), \dots, (a^i_{r_i},b^i_{r_i}))$ for $i = 1, \dots, k$. Notice that $a^i_1 = 1$ and $b^{i}_{r_i} = -1$ for $i = 1, \dots, k$, and also from (\ref{Eq.Layers1})
\begin{equation*}
a_{r_1}^1 < a_{r_2}^2 < \cdots < a_{r_k}^k ,\mbox{ and } b_{1}^1 > b_{1}^2 > \cdots > b_{1}^k. 
\end{equation*}
In particular, we see that $s_i$ contains a vertex of the form $(i, u_i)$ for each $i = 1, \dots, k$, and we pick $u_i$ maximal with respect to this condition. Observe that 
\begin{equation}\label{Eq.OD4}
u_1 \geq u_2 \geq \cdots \geq u_k. 
\end{equation}
Indeed, if $i \leq k-1$, then $u_i \geq u_{i+1}$ is automatically satisfied if $u_i = -1$. If $u_i < -1$, then by maximality of $u_i$, we must have $(i+1, u_i) \in s_i = \partial \lambda^i \subseteq \mathring{\lambda}^{i+1}$. As $(i+1, u_{i+1}) \in \partial \lambda^{i+1}$, we conclude $u_i \geq u_{i+1}$ in this case as well.

Lastly, consider the pairwise disjoint up-right paths $\tilde{s}_1, \dots, \tilde{s}_k$ obtained as follows. Each $\tilde{s}_i$ is obtained from $s_i$ by deleting all vertices that come before $(i, u_i)$ and adding the vertices $(i,u_i-1), (i, u_i-2), \dots, (i, -n)$, see the right side of Figure \ref{Fig.Rectangle}. We observe that $s_i \cap \{(a,b): a \geq i+1\} = \tilde{s}_i \cap \{(a,b): a \geq i+1\}$. In addition, if $(a,b) \in s_i$ with $a \leq i$, we have $b \leq u_i$ (by maximality of $u_i$ and the up-right path condition), and so $b \leq u_a$ by (\ref{Eq.OD4}). In particular, $(a,b) \in \tilde{s}_a$. The last two observations show that $\sqcup_{i = 1}^k s_i \subseteq \sqcup_{i = 1}^k \tilde{s}_i,$
which by (\ref{Eq.OD1}) and the maximality of $(\chi_1, \dots, \chi_k)$ implies 
\begin{equation*}
\sqcup_{i = 1}^k s_i =  \sqcup_{i = 1}^k \tilde{s}_i.
\end{equation*}
Note that from (\ref{Eq.Layers1}), the sequence $(i,u_i), (i,u_i-1), \dots, (i, -n)$ contains vertices from $s_i, s_{i+1}, \dots, s_k$. This means $n + u_i + 1 \geq k -i + 1$ or $u_i \geq -n + k-i $. Consequently, $(j_0, -n + k - j_0) \in \tilde{s}_{j_0} \subseteq \sqcup_{i = 1}^k s_i$, leading to the desired contradiction. This concludes the proof of (\ref{Eq.OD2}) and hence the lemma.
\end{proof}

%
%
\subsection{Proof of Lemma \ref{L.GequalsH}}\label{SectionA.3} Suppose first that $k \geq \min(m,n) + 1$. If $m \leq n$, we set $\chi_i = ((i,1), \dots, (i,n))$ for $i = 1, \dots, m$ and $\chi_i = \emptyset$ for $i = m+1, \dots, k$, and if $m > n$, we set $\chi_i = ((1,i), \dots, (m,i))$ for $i = 1, \dots, n$ and $\chi_i = \emptyset$ for $i = n+1, \dots, k$. From (\ref{Eq.GDef2}) and (\ref{Eq.HDef}) we get
$$h_k(A) \geq A(\chi_1) + \cdots + A(\chi_k) = \sum_{i = 1}^m \sum_{j = 1}^n a_{i,j} = g_k(A).$$
On the other hand, as $a_{i,j} \geq 0$, we have for any pairwise disjoint NE-chains $(s_1, \dots, s_k)$ 
$$A(s_1) + \cdots + A(s_k) \leq \sum_{i = 1}^m \sum_{j = 1}^n a_{i,j} = g_k(A),$$
and so $h_k(A) \leq g_k(A)$ as well. This proves (\ref{Eq.GHEqual}) when $k \geq \min(m,n)+ 1$.

In the remainder we assume $k \leq \min(m,n)$. Since up-right paths are NE-chains, we have
\begin{equation}\label{Eq.GHIneq1}
h_k(A) \geq g_k(A).
\end{equation}
Suppose that $\chi_1, \dots, \chi_k$ are pairwise disjoint NE-chains such that
$$h_k(A) = A(\chi_1) + \cdots + A(\chi_k).$$
Let $\tilde{\chi}_1, \dots, \tilde{\chi}_k$ be pairwise disjoint NE-chains such that 
$$\sqcup_{i = 1}^k \chi_i \subseteq \sqcup_{i = 1}^k \tilde{\chi}_i,$$
and $\cup_{i = 1}^k \tilde{\chi}_i$ is maximal with respect to set inclusion. As the entries of $A$ are non-negative, we have 
\begin{equation}\label{Eq.HWithMax}
h_k(A) = A(\tilde{\chi}_1) + \cdots + A(\tilde{\chi}_k).
\end{equation}

From Lemma \ref{Lem.OffDiag} we have that 
\begin{equation}\label{Eq.PropChains}
(j, k-j+1) \in \sqcup_{i = 1}^k \tilde{\chi}_i \mbox{, and }  (m-j+1, n-k + j) \in \sqcup_{i = 1}^k \tilde{\chi}_i.
\end{equation}
In addition, from Lemma \ref{Lem.Layers} and the maximality of $\cup_{i = 1}^k \tilde{\chi}_i$, we can find $k$ pairwise disjoint up-right paths $(\pi_1, \dots, \pi_k)$ (these are the $\partial \lambda^i$), such that 
\begin{equation}\label{Eq.PropPaths}
\sqcup_{i =1}^k \tilde{\chi}_i = \sqcup_{i = 1}^k \pi_i.
\end{equation}
Notice that no two vertices $(j, k-j+1)$ can belong to the same path $\pi_i$, and so precisely one vertex belongs to each path, showing that $\pi_i \neq \emptyset$ for all $i = 1, \dots, k$. Setting $\pi_i = ((a^i_1, b^i_1), \dots, (a^i_{r_i}, b^i_{r_i}))$, we have from Lemma \ref{Lem.Layers} that $a^i_1 = 1$, $b^{i}_{r_i} = n$ and
\begin{equation}\label{Eq.PropPaths2}
a^1_{r_1} < a_{r_2}^2 < \cdots < a^k_{r_k}, \mbox{ and } b_{1}^1 > b_{1}^2 > \cdots > b_{1}^k.
\end{equation}
From the second set of inequalities in (\ref{Eq.PropPaths2}) we have for $(a,b) \in \pi_i$ that $b \geq b_{1}^i \geq k-i+1$. From the first part of (\ref{Eq.PropChains}) and (\ref{Eq.PropPaths}), we know that $(k,1)$ belongs to some $\pi_i$, and the last observation shows that this is only possible if $b_1^k = 1$, $(k,1) \in \pi_k$. As $a^k_1 = 1$, we must in fact have $(1,1), \dots, (k,1) \in \pi_k$. Looking at $(k-1,2)$, we have that it can only belong to $\pi_k$ and $\pi_{k-1}$, and as it cannot belong to $\pi_k$ (as it contains $(k,1)$), we must have $b_1^{k-1} = 2$, $(k-1,2) \in \pi_{k-1}$. As $a^{k-1}_1 = 1$, we must in fact have $(1,2), (2,2), \dots, (k-1,2) \in \pi_{k-1}$. Iterating this argument, we conclude that $b_1^{i} = k-i+1$ and $(1,k-i+1), (2,k-i+1), \dots, (i,k-i+1) \in \pi_{i}$ for $i =1, \dots, k$.

If instead we use the first set of inequalities in (\ref{Eq.PropPaths2}), the second part of (\ref{Eq.PropChains}), and the fact that $b^{i}_{r_i} = n$ for $i = 1, \dots, k$, we similarly conclude that $a^i_{r_i} = m-k + i$ and $(m-k+i,n - i + 1), (m-k+i,n-i + 2), \dots, (m-k+i,n) \in \pi_{i}$ for $i = 1, \dots, k$. The work so far shows that the paths $\pi_1, \dots, \pi_k$ are pairwise disjoint, and $\pi_i$ connects $(1,k-i+1)$ in a horizontal straight line to $(i,k-i+1)$, then connects $(i,k-i+1)$ in some way to $(m-k+i,n - i + 1)$ and goes vertically up to $(m-k+i, n)$.

\begin{figure}[ht]
\centering
\begin{tikzpicture}[scale=0.58]

\def\m{7} 
\def\n{5} 
\def\k{3} 

\definecolor{Bg}{gray}{1.0}        
\definecolor{C1}{gray}{0.8} 
\definecolor{C2}{gray}{0.5} 
\definecolor{C3}{gray}{0.2} 

\begin{scope}[shift={(0,0)}]


\foreach \x/\y in {5/1, 4/1, 3/1, 2/1, 2/2, 1/2, 1/3} {
  \fill[C1] (\x,-\y) rectangle ++(1,-1);
}

\foreach \x/\y in {6/1, 6/2, 5/2, 5/3, 4/3, 3/3, 3/4, 2/4, 1/4} {
  \fill[C2] (\x,-\y) rectangle ++(1,-1);
}

\foreach \x/\y in {7/1, 7/2, 7/3, 6/3, 6/4, 5/4, 4/4, 4/5, 3/5, 2/5, 1/5} {
  \fill[C3] (\x,-\y) rectangle ++(1,-1);
}

\foreach \i in {1,...,\n}{
  \foreach \j in {1,...,\m}{
    \draw[black] (\j,-\i ) rectangle ++(1,-1);
  }
}

   \begin{scope}[shift={(0,-2)}]
  \draw[fill=C1] (-1,0) rectangle ++(0.5,-0.5);
  \node[anchor=west] at (-0.5,-0.3) {$\pi_1$};
  \draw[fill=C2] (-1,-1) rectangle ++(0.5,-0.5);
  \node[anchor=west] at (-0.5,-1.3) {$\pi_2$};
  \draw[fill=C3] (-1,-2) rectangle ++(0.5,-0.5);
  \node[anchor=west] at (-0.5,-2.3) {$\pi_3$};
  \end{scope}

\draw[fill=white, line width=0.5mm] (5.5,-1.5) circle (4pt);
\draw[fill=white, line width=0.5mm] (6.5,-2.5) circle (4pt);
\draw[fill=white, line width=0.5mm] (7.5,-3.5) circle (4pt);

\end{scope}

\begin{scope}[shift={(11,0)}]


\foreach \x/\y in {7/1, 6/1, 5/1, 4/1, 3/1, 2/1, 2/2, 1/2, 1/3} {
  \fill[C1] (\x,-\y) rectangle ++(1,-1);
}

\foreach \x/\y in {7/2, 6/2, 5/2, 5/3, 4/3, 3/3, 3/4, 2/4, 1/4} {
  \fill[C2] (\x,-\y) rectangle ++(1,-1);
}

\foreach \x/\y in {7/3, 6/3, 6/4, 5/4, 4/4, 4/5, 3/5, 2/5, 1/5} {
  \fill[C3] (\x,-\y) rectangle ++(1,-1);
}

\foreach \i in {1,...,\n}{
  \foreach \j in {1,...,\m}{
    \draw[black] (\j,-\i ) rectangle ++(1,-1);
  }
}

\draw[fill=white, line width=0.5mm] (5.5,-1.5) circle (4pt);
\draw[fill=white, line width=0.5mm] (6.5,-2.5) circle (4pt);
\draw[fill=white, line width=0.5mm] (7.5,-3.5) circle (4pt);

   \begin{scope}[shift={(0,-2)}]
  \draw[fill=C1] (-1,0) rectangle ++(0.5,-0.5);
  \node[anchor=west] at (-0.5,-0.3) {$\tilde{\pi}_1$};
  \draw[fill=C2] (-1,-1) rectangle ++(0.5,-0.5);
  \node[anchor=west] at (-0.5,-1.3) {$\tilde{\pi}_2$};
  \draw[fill=C3] (-1,-2) rectangle ++(0.5,-0.5);
  \node[anchor=west] at (-0.5,-2.3) {$\tilde{\pi}_3$};
  \end{scope}

\end{scope}

\end{tikzpicture}
\caption{The figure depicts the up-right paths $\pi_1, \dots, \pi_k$ and $\tilde{\pi}_1, \dots, \tilde{\pi}_k$. The circles indicate the locations of the points $(m-k+i, n-i+1)$ for $i = 1,\dots, k$.} \label{Fig.SwapEnds}
\end{figure}

Note that we can modify the paths $\pi_1, \dots, \pi_k$ to $\tilde{\pi}_1, \dots, \tilde{\pi}_k$, so that $\pi_i$ agrees with $\tilde{\pi}_i$ up to $(m-k+i,n - i + 1)$, but afterwards instead of going vertically up, goes horizontally to the right, see Figure \ref{Fig.SwapEnds}. One readily observes that $\sqcup_{i = 1}^k \pi_i =  \sqcup_{i = 1}^k \tilde{\pi}_i$, which together with (\ref{Eq.HWithMax}), (\ref{Eq.PropPaths}) shows
\begin{equation}\label{Eq.GHIneq2}
h_k(A) = A(\tilde{\chi}_1) + \cdots + A(\tilde{\chi}_k) \leq A(\tilde{\pi}_1) + \cdots + A(\tilde{\pi}_k) \leq g_k(A).
\end{equation}
In the last inequality we used the definition of (\ref{Eq.GDef}) and the fact that $\tilde{\pi}_1, \dots, \tilde{\pi}_k$ are pairwise disjoint and $\tilde{\pi}_i$ connects $(1,k-i+1)$ to $(m,n - i + 1)$. Equations (\ref{Eq.GHIneq1}) and (\ref{Eq.GHIneq2}) prove (\ref{Eq.GHEqual}) when $k \leq \min(m,n)$.

%
%
\section{Exponential LPP}\label{SectionB} The goal of this section is to establish analogues of Theorems \ref{T1} and \ref{T2} for exponential last passage percolation by taking a suitable diffuse limit.

%
%
\subsection{Basic notation}\label{SectionB.1} In this section, we introduce notation that will be used to state the results for exponential LPP.

For $\alpha>0$, the exponential distribution $\mathrm{Exp}(\alpha)$ is supported on $\mathbb{R}_{>0}$ with density $\alpha e^{-\alpha x}$ against the Lebesgue measure.

The full-space exponential LPP depends on two sequences of real parameters $\{\alpha_i\}_{i \geq 1}$, $\{\beta_i\}_{i \geq 1}$, such that $\alpha_i+\beta_j >0$ for all $i,j \geq 1$. The background noise is given by an array $E = (e_{i,j}:i,j \geq 1)$ of independent exponential variables $e_{i,j} \sim \mathrm{Exp}(\alpha_i+\beta_j)$.
We define the higher-rank last passage times by
\begin{equation}\label{Eq.ELPTk}
G_k^{\exp}(m,n) := g_k(E[m,n]) = h_k(E[m,n]),
\end{equation}
where $k \geq 1$, $E[m,n] = (e_{i,j}: i = 1, \dots, m, j = 1, \dots, n)$, the deterministic functions $g_k$ and $h_k$ are as in \eqref{Eq.GDef},  \eqref{Eq.GDef2} and \eqref{Eq.HDef}, and the last equality follows by Lemma \ref{L.GequalsH}.

The half-space exponential LPP depends on a sequence of real parameters $\{\alpha_i\}_{i \geq 1}$ and a parameter $\alpha_{\circ}$ such that $\alpha_i+\alpha_j >0$ and $\alpha_{\circ}+\alpha_i >0$ for all $i,j \geq 1$. The background noise is given by an array $E = (e_{i,j}: i,j \geq 1)$, where $(e_{i,j}: 1\leq i \leq j)$ are independent exponential variables with $e_{i,j} \sim \mathrm{Exp}(\alpha_i+\alpha_j)$ when $i \neq j$ and $e_{i,i} \sim \mathrm{Exp}(\alpha_{\circ}+\alpha_i)$, and $e_{i,j} = e_{j,i}$ for all $i,j \geq 1$.  
We introduce the same last passage times as in \eqref{Eq.ELPTk}. Note that 
$G_k^{\exp}(m,n) = G_k^{\exp}(n,m)$ for all $k,m,n \geq 1$ by the same argument as below (\ref{Eq.SymmG}).

For $r \in \mathbb{Z}_{\geq 0}$, we define
\[ 
\mathbb W_r=\{\phi=(\phi_1,\phi_2,\dots)\in \mathbb R^\infty: \phi_1\geq \phi_2 \geq \cdots \mbox{, and } 0 = \phi_{r+1} = \phi_{r+2} = \cdots \},
\]
and set $\mathbb{W}_{\infty}=\cup_{r=0}^{\infty}\mathbb W_r$. In words, $\mathbb{W}_{\infty}$ consists of all non-strictly decreasing sequences of non-negative real numbers that are eventually zero, and $\mathbb{W}_r$ is the subset of all such sequences that contain at most $r$ positive terms. Notice that $\mathbb{W}_r \subset \mathbb{W}_{r'}$ when $r \leq r'$, and $\mathbb{W}_0$ consists of a single element that we denote by $\varnothing$. In addition, by projecting onto the first $r$ coordinates, we have the following natural identification for all $r \in \mathbb{Z}_{\geq 1}$
\begin{equation}\label{Eq.Isomorphism}
\mathbb{W}_r \cong \{x = (x_1, \dots, x_r) \in \mathbb{R}^r: x_1 \geq \cdots \geq x_r \geq 0 \} \subset \mathbb{R}^r.
\end{equation} 

Given $\phi,\psi\in\mathbb{W}_{\infty}$, we say that they {\em interlace} (denoted $\phi \succeq \psi$ or $\psi \preceq \phi$) if $\phi_1 \geq \psi_1 \geq \phi_2 \geq \psi_2 \geq  \cdots$.  
For $\phi,\psi\in\mathbb{W}_{\infty}$ and $\alpha\in\mathbb{R}$, we define
\begin{equation}\label{Eq.PDef} 
p_{\phi/\psi}(\alpha) = {\bf 1}\{\phi \succeq \psi\} \cdot e^{-\alpha \sum_{i\geq1}(\phi_i-\psi_i) }. 
\end{equation} 
For $m,n\geq1$, we define
\begin{equation}\label{Eq.Defphi}
\begin{split}
&\phi_{1}(m,n) = G_1^{\exp}(m,n), \mbox{ and } \\
&\phi_k(m,n) = G_k^{\exp}(m,n) - G_{k-1}^{\exp}(m,n) \mbox{ for } k \geq 2,
\end{split}
\end{equation}
Note that by \eqref{Eq.GDef},  \eqref{Eq.GDef2}, and (\ref{Eq.ELPTk}), we have $\phi(m,n):=(\phi_{k}(m,n): k \geq 1)\in\mathbb{W}_{\min(m,n)}$.

We end this section with the following elementary result.
\begin{lemma}\label{lem:scaling}
Fix $\alpha>0$. Let $w^{(\varepsilon)}\sim\mathrm{Geom}(e^{-\varepsilon\alpha})$ for each $\varepsilon>0$. Then, as $\varepsilon\to0+$, the law of $\varepsilon w^{(\varepsilon)}$ converges weakly to $\mathrm{Exp}(\alpha)$. 
\end{lemma}
\begin{proof}
For every $u\geq 0$, we have $\mathbb P ( \varepsilon w^{(\varepsilon)}>u )
=e^{-\varepsilon\alpha(\lfloor u/\varepsilon\rfloor+1)}\to e^{-\alpha u}$ as $\varepsilon\to0+$.  
\end{proof}

%
%
\subsection{Full-space exponential LPP}\label{SectionB.2} We adopt the same notation as in Sections \ref{Section2.3}, \ref{SectionA.1}, and \ref{SectionB.1}. The following result is the exponential LPP analogue of Theorem \ref{T1}.

\begin{theorem}\label{T1exp}
Suppose that $E$ is the background noise of the full-space exponential LPP model as in Section \ref{SectionB.1}, and $(\phi(m,n): m,n \geq 1)$ are as in \eqref{Eq.Defphi}. In addition, set $\phi(m,n) = \varnothing$ when $m = 0$ or $n = 0$. Finally, fix $M, N \geq 0$ and a down-right path $\gamma = (v_0, \dots, v_{M+N})$ from $(0,N)$ to $(M,0)$, whose corresponding word is $s_1s_2\cdots s_{M+N}$. Write $v_i=(m_i,n_i)$ and $r_i=\min(m_i,n_i)$ for $0\leq i\leq M+N$. Then, for every Borel set  $A\subseteq \prod_{i=0}^{M+N}\mathbb W_{r_i}$,
\begin{equation}\label{Eq.FSDistrexp}
\mathbb{P}\left(\left(\phi(v_i)\right)_{i=0}^{M+N}\in A\right) = Z \cdot\int_A  \prod_{i = 1}^{M+N} H(\phi^{i-1}, \phi^{i})\cdot   d\phi^0\cdots d\phi^{M+N} ,   
\end{equation}
where $d\phi^i$ denotes Lebesgue measure on $\mathbb W_{r_i}$ if $r_i\geq 1$ and unit mass at $\varnothing$ if $r_i=0$, $Z = \prod_{(a,b) \in Y(\gamma)} (\alpha_a+\beta_b)$,  and
\begin{equation}\label{Eq.DefH} 
H(\phi^{i-1}, \phi^{i}) = \begin{cases}  p_{\phi^{i-1}/\phi^{i}}(\beta_t) &\mbox{ if }  v_{i-1}\in \mathbb Z\times\{t\} \text{ and } s_i=D, \\  p_{\phi^i/\phi^{i-1}}(\alpha_t) & \mbox{ if } v_{i-1}\in \{t-1\}\times\mathbb Z
\text{ and } s_i=R. \end{cases}
\end{equation}
\end{theorem}
\begin{proof}
    For $\varepsilon>0$ and $i \geq 1$, set $x_i^{(\varepsilon)}=e^{-\varepsilon\alpha_i}$ and $y_i^{(\varepsilon)}=e^{-\varepsilon\beta_i}$. Then $x_i^{(\varepsilon)},y_j^{(\varepsilon)}>0$ and $x_i^{(\varepsilon)}y_j^{(\varepsilon)}
 =e^{-\varepsilon(\alpha_i+\beta_j)}\in(0,1)$ for all $i,j\geq1$. Let $W^{(\varepsilon)}=(w_{i,j}^{(\varepsilon)}:i,j\geq 1)$ be the corresponding geometric environment as in Section \ref{Section2.1} and $\lambda^{(\varepsilon)}(m,n)$ be the partitions constructed from it as in (\ref{Eq.DefLambda}). Set also $\lambda^{(\varepsilon)}(m,n)=\emptyset$ when $m=0$ or $n=0$.

As random vectors in  $\prod_{i=0}^{M+N}\mathbb W_{r_i}$, we have
\begin{equation}\label{eq:convergence partitions geo to exp}
\left(\varepsilon\lambda^{(\varepsilon)}(v_i)\right)_{i=0}^{M+N}\Rightarrow \left(\phi(v_i)\right)_{i=0}^{M+N}\quad\mbox{as }\varepsilon\to0+.
\end{equation}
Indeed, only $w_{i,j}^{(\varepsilon)}$ for $1\leq i\leq M$ and $1\leq j\leq N$ are involved on the left side of \eqref{eq:convergence partitions geo to exp}, hence Lemma \ref{lem:scaling} and independence give the joint convergence of the scaled geometric weights in this rectangle to the corresponding exponential weights. The maps $g_k$ for $k\geq1$ are continuous on this rectangle and satisfy $g_k(cA)=c g_k(A)$ for every $c\geq0$. Hence the continuous mapping theorem allows us to conclude \eqref{eq:convergence partitions geo to exp}. 

Theorem \ref{T1} gives
\begin{equation}\label{eq:geoscaled}
    \mathbb P\left( \varepsilon\lambda^{(\varepsilon)}(v_i) = \varepsilon\lambda^i \mbox{ for }0\leq i\leq M+N\right)
=Z_\varepsilon\prod_{i=1}^{M+N}H(\varepsilon\lambda^{i-1},\varepsilon\lambda^i), 
\end{equation}
and the measure is supported on all partitions $\lambda^i$ satisfying $\ell(\lambda^i)\leq r_i$ for $0\leq i\leq M+N$. Here, $Z_\varepsilon=\prod_{(a,b)\in Y(\gamma)}
 \left(1-e^{-\varepsilon(\alpha_a+\beta_b)}\right)$. Indeed, the requirements $\ell(\lambda^i)\le r_i$ for $0\le i\le M+N$ follow from \eqref{Eq.GDef} and \eqref{Eq.GDef2}, and \eqref{eq:geoscaled} follows from $s_{\lambda /\mu}(e^{-\varepsilon\alpha})=p_{\varepsilon\lambda/\varepsilon\mu}(\alpha)$ for each $\alpha\in\mathbb R$ and $\varepsilon>0$.

We next show that 
\begin{equation}\label{eq:scalinglimitnormalization}
    Z_{\varepsilon}\cdot\varepsilon^{-\sum_{i=0}^{M+N}r_i}\to Z\quad\mbox{as }\varepsilon\to0+.
\end{equation}
Indeed, using the definitions of $Z$ and $Z_{\varepsilon}$, and the fact that for each $x\in\mathbb{R}$, $(1-e^{-\varepsilon x})\varepsilon^{-1}\to x$ as $\varepsilon\to0+$, it suffices to show that $|Y(\gamma)|=\sum_{i=0}^{M+N}r_i$. This identity follows by noting that $(i,\ell)\mapsto (m_i-\ell+1,n_i-\ell+1)$ is a bijection between $0\leq i\leq M+N$, $1\leq \ell\leq r_i$ and $Y(\gamma)$.
 
The space $\prod_{i=0}^{M+N}\mathbb W_{r_i}$ equipped with the product measure $d\phi^0\cdots d\phi^{M+N}$ can be identified with a closed subspace of $\mathbb R^{\sum_{i=0}^{M+N}r_i}$ with the Lebesgue measure using (\ref{Eq.Isomorphism}). 
Let $\nu$ be the law of $\left(\phi(v_i)\right)_{i=0}^{M+N}$ and let $\mu$ be the Borel measure on $\prod_{i=0}^{M+N}\mathbb W_{r_i}$ whose density with respect to $d\phi^0\cdots d\phi^{M+N}$ is $Z \cdot \prod_{i=1}^{M+N}H(\phi^{i-1},\phi^i)$. The conclusion \eqref{Eq.FSDistrexp} of this theorem is equivalent to $\nu=\mu$, which by the local compactness of Euclidean spaces is equivalent to showing that 
\begin{equation}\label{eq:Fnumu}
\int_{\prod_{i=0}^{M+N}\mathbb W_{r_i}}F\,d\nu=\int_{\prod_{i=0}^{M+N}\mathbb W_{r_i}}F\,d\mu
\end{equation}
for every continuous compactly supported function $F$ on $\prod_{i=0}^{M+N}\mathbb W_{r_i}$. 

For such a function $F$, the function
\[f(\phi^0,\dots,\phi^{M+N})= F(\phi^0,\dots,\phi^{M+N})
  \prod_{i=1}^{M+N}H(\phi^{i-1},\phi^i)
\]
is bounded and compactly supported.  After extending it by zero to $\mathbb R^{\sum_{i=0}^{M+N}r_i}$, its discontinuities are contained in the boundary hyperplanes of the Weyl chambers in (\ref{Eq.Isomorphism}) and in the finitely many hyperplanes where one of the interlacing inequalities in a factor $H$ becomes an equality.  This set has Lebesgue measure zero, and hence $f$ is Riemann integrable on
$\mathbb R^{\sum_{i=0}^{M+N}r_i}$, and the lattice Riemann sums satisfy
\begin{multline*} 
  \varepsilon^{\sum_{i=0}^{M+N}r_i}\sum_{ \lambda^i\in\mathbb{Y},\mbox{ } \ell(\lambda^i)\leq r_i,\mbox{ } 0\leq i\leq M+N   }
f(\varepsilon\lambda^0,\dots,\varepsilon\lambda^{M+N}) 
 \to\\
 \int_{\prod_{i=0}^{M+N}\mathbb W_{r_i}} f(\phi^0,\dots,\phi^{M+N})\,
 d\phi^0\cdots d\phi^{M+N} ,
\end{multline*}
as $\varepsilon\to0+$. Combining the above with \eqref{eq:scalinglimitnormalization}, we get
\begin{align*}
&\mathbb E\left[F\left( \varepsilon\lambda^{(\varepsilon)}(v_0),\dots,\varepsilon\lambda^{(\varepsilon)}(v_{M+N})\right)  \right]\\
&=Z_\varepsilon\cdot
  \sum_{ \lambda^i\in\mathbb{Y},\mbox{ } \ell(\lambda^i)\leq r_i,\mbox{ } 0\leq i\leq M+N }
  f(\varepsilon\lambda^0,\dots,\varepsilon\lambda^{M+N}) \\ 
&\to
Z\cdot\int_{\prod_{i=0}^{M+N}\mathbb W_{r_i}}F(\phi^0,
\dots,\phi^{M+N})
\prod_{i=1}^{M+N}H(\phi^{i-1},\phi^i)\,
 d\phi^0\cdots d\phi^{M+N} ,
\end{align*}
as $\varepsilon\to0+$. 
Combining the latter with \eqref{eq:convergence partitions geo to exp}, we conclude \eqref{eq:Fnumu}.
\end{proof}

%
%
\subsection{Half-space exponential LPP}\label{SectionB.3} We adopt the same notation as in Sections \ref{Section2.3}, \ref{SectionA.1}, and \ref{SectionB.1}. The following result is the exponential LPP analogue of Theorem \ref{T2}.

\begin{theorem}\label{T2exp}
Suppose that $E$ is the background noise of the half-space exponential LPP model as in Section \ref{SectionB.1}, and $(\phi(m,n): m,n \geq 1)$ are as in \eqref{Eq.Defphi}. In addition, set $\phi(m,n) = \varnothing$ when $m = 0$ or $n = 0$. Finally, fix any $M, N \geq 0$ and a down-right path $\gamma = (v_0, \dots, v_{M+N})$ from $(N,N)$ to $(M+N,0)$, whose corresponding word is $s_1s_2\cdots s_{M+N}$. 
Write $v_i=(m_i,n_i)$ for $0\leq i\leq M+N$. Then, for every Borel set  $A\subseteq \prod_{i=0}^{M+N}\mathbb W_{n_i}$,
\begin{equation}\label{Eq.HSDistrexp}
\mathbb{P}\left(\left(\phi(v_i)\right)_{i=0}^{M+N}\in A\right) =  Z \cdot\int_A  \kappa_{\phi^0}(\alpha_{\circ}) \cdot \prod_{i = 1}^{M+N} \hspace{-1mm} H(\phi^{i-1}, \phi^{i})\cdot   d\phi^0\cdots d\phi^{M+N} ,   
\end{equation}
where $H(\phi^{i-1}, \phi^{i})$ are as in \eqref{Eq.DefH} with $\beta_i = \alpha_i$, $\kappa_{\phi}(\alpha) = e^{-\alpha(\phi_1 - \phi_2 + \phi_3 - \phi_4 + \cdots)}$, $d\phi^i$ denotes Lebesgue measure on $\mathbb W_{n_i}$ if $n_i\geq 1$ and unit mass at $\varnothing$ if $n_i=0$, and $Z$ is a normalization constant. If we denote $Y_{>}(\gamma) = \{(a,b) \in Y(\gamma): a > b \}$ with $Y(\gamma)$ as in \eqref{Eq.FerresShapeDef}, then 
$$Z = \prod_{i = 1}^N (\alpha_{\circ}+\alpha_i) \cdot \prod_{(a,b) \in Y_{>}(\gamma)} (\alpha_a+\alpha_b).$$
\end{theorem}
\begin{proof} The proof is similar to that of Theorem \ref{T1exp}. 

For $\varepsilon>0$, set $c^{(\varepsilon)}=e^{-\varepsilon\alpha_{\circ}}$ and $x_i^{(\varepsilon)}=e^{-\varepsilon\alpha_i}$ for $i\geq1$. Then $x_i^{(\varepsilon)}>0$, $c^{(\varepsilon)}x_i^{(\varepsilon)}=e^{-\varepsilon(\alpha_{\circ}+\alpha_i)}\in(0,1)$, and 
$x_i^{(\varepsilon)}x_j^{(\varepsilon)}=e^{-\varepsilon(\alpha_i+\alpha_j)}\in(0,1)$ for $i,j\geq1$. Let $W^{(\varepsilon)}=(w_{i,j}^{(\varepsilon)}:i,j\geq 1)$ be the symmetrized geometric environment as in Section \ref{Section2.2} for the above parameters, let $\lambda^{(\varepsilon)}(m,n)$ be the corresponding partitions for $m,n\geq1$ as in (\ref{Eq.DefLambda}), and set $\lambda^{(\varepsilon)}(m,n)=\emptyset$ when $m=0$ or $n=0$.
As random vectors in  $\prod_{i=0}^{M+N}\mathbb W_{n_i}$, we have
\begin{equation}\label{eq:convergence partitions geo to exp half space}
\left(\varepsilon\lambda^{(\varepsilon)}(v_i)\right)_{i=0}^{M+N}\Rightarrow \left(\phi(v_i)\right)_{i=0}^{M+N}\quad\mbox{as }\varepsilon\to0+.
\end{equation}
Indeed, the left side of \eqref{eq:convergence partitions geo to exp half space} depends on $w_{i,j}^{(\varepsilon)}$ for $1\leq i\leq M+N$ and $1\leq j\leq N$, hence by half-space symmetry, it is a function of the independent variables $w_{i,j}^{(\varepsilon)}$ for $(i,j)\in\mathcal{I}$, where $\mathcal{I}=\{(i,j):1\leq j\leq N,\ j\leq i\leq M+N\}$. 
Lemma \ref{lem:scaling} and independence give the joint convergence of the scaled geometric weights $\varepsilon w_{i,j}^{(\varepsilon)}$ for $(i,j)\in\mathcal{I}$ to the corresponding independent exponential variables of the half-space environment $E$. The maps $g_k$ for $k\geq1$ are continuous on the environment variables indexed by $\mathcal I$ and satisfy $g_k(cA)=c g_k(A)$ for every $c\geq0$. Hence the continuous mapping theorem allows us to conclude \eqref{eq:convergence partitions geo to exp half space}.  

Theorem \ref{T2} gives 
\begin{equation}\label{eq:geoscaled half space}
\mathbb P\left(\hspace{-0.5mm}
\varepsilon\lambda^{(\varepsilon)}(v_i)=\varepsilon\lambda^i
\mbox{ for }0\leq i\leq M+N \hspace{-0.5mm}
\right)  =Z_{\varepsilon}\,
\kappa_{\varepsilon\lambda^0}(\alpha_{\circ}) \hspace{-2mm}
\prod_{i=1}^{M+N}\hspace{-2mm}
H(\varepsilon\lambda^{i-1},\varepsilon\lambda^i),
\end{equation}
and the measure is supported on all partitions $\lambda^i$ satisfying $\ell(\lambda^i)\leq n_i$ for $0\leq i\leq M+N$. Here, $Z_{\varepsilon}=\prod_{j=1}^N\left(1-e^{-\varepsilon(\alpha_{\circ}+\alpha_j)}\right)\prod_{(a,b)\in Y_{>}(\gamma)}\left(1-e^{-\varepsilon(\alpha_a+\alpha_b)}\right)$. Indeed, the conditions $\ell(\lambda^i)\leq n_i$ for $0\leq i\leq M+N$ follow from (\ref{Eq.GDef}) and (\ref{Eq.GDef2}), while \eqref{eq:geoscaled half space} follows from $\tau_{\lambda }(e^{-\varepsilon\alpha })=\kappa_{\varepsilon\lambda }(\alpha )$ and $s_{\lambda/\mu}(e^{-\varepsilon \alpha})=p_{\varepsilon\lambda/\varepsilon\mu}(\alpha)$ for $\alpha\in\mathbb{R}$, $\varepsilon>0$. 

We next show that
\begin{equation}\label{eq:scalinglimitnormalization half space}
    Z_{\varepsilon}\cdot\varepsilon^{-\sum_{i=0}^{M+N}n_i}\to Z\quad\mbox{as }\varepsilon\to0+.
\end{equation}
Indeed, using the fact that for each $x\in\mathbb{R}$, $(1-e^{-\varepsilon x})\varepsilon^{-1}\to x$ as $\varepsilon\to0+$, it suffices to show that $N+|Y_{>}(\gamma)|=\sum_{i=0}^{M+N}n_i$. This identity follows by noting that $(i,\ell)\mapsto (m_i-\ell+1,n_i-\ell+1)$ is a bijection between $0\leq i\leq M+N$, $1\leq \ell\leq n_i$ and $Y_{>}(\gamma)\sqcup\{(1,1),\dots,(N,N)\}$.   

The rest of the proof is the same as the proof of Theorem \ref{T1exp} starting from the paragraph above \eqref{eq:Fnumu}, except for changing every $r_i$ to $n_i$, the density of $\mu$ to $Z\cdot\kappa_{\phi^0}(\alpha_{\circ})\cdot\prod_{i=1}^{M+N}H(\phi^{i-1},\phi^i)$, and the function $f$ to  
$f(\phi^0,\dots,\phi^{M+N})
=
F(\phi^0,\dots,\phi^{M+N})
\kappa_{\phi^0}(\alpha_{\circ})
\prod_{i=1}^{M+N}H(\phi^{i-1},\phi^i)$.
\end{proof}

\end{appendix}

\section*{Acknowledgments}
The authors would like to thank the American Institute of Mathematics and the organizers, Leonid Petrov and Axel Saenz Rodriguez, of the AIM Workshop ``All roads to the {K}{P}{Z} universality class'', where this project was initiated. We are also grateful to Guillaume Barraquand and Ivan Corwin for their useful feedback on earlier drafts of the paper. Finally, we thank the anonymous referees for their careful reading of the manuscript and for numerous helpful comments and suggestions that improved its presentation.

\section*{funding}
ED was partially supported by the Simons Award TSM-00014004 from the Simons Foundation International. ZY was partially supported by the Miller Institute for Basic Research in Science at the University of California, Berkeley, Ivan Corwin's NSF grants DMS-1811143, DMS-2246576, Simons Foundation Grant 929852, and the Fernholz Foundation's ``Summer Minerva Fellows'' program.

\bibliographystyle{amsplain} 
\bibliography{PD}

\end{document}